\documentclass[12pt]{amsart}

\usepackage{amsmath}
\usepackage{amssymb}
\usepackage{amsthm}
\usepackage{graphicx}
\usepackage[caption=false]{subfig}
\usepackage{enumerate}
\usepackage[numbers,sort&compress]{natbib}
\usepackage[pagewise]{lineno}
\usepackage{etoolbox}

\theoremstyle{plain}
\newtheorem{theorem}{Theorem}[section]
\newtheorem{lemma}[theorem]{Lemma}

\theoremstyle{definition}
\newtheorem{definition}[theorem]{Definition}
\newtheorem{example}[theorem]{Example}

\theoremstyle{remark}

\newcommand\mystyle{\everymath{\displaystyle}}
\mystyle

\title[Nonequilibrium transient centers]
      {Nonequilibrium transient centers in continuous-time systems}

\author[A. Pasion]{Anthony Pasion$^{1,2}$}
\author[F. Magpantay]{Felicia Magpantay$^{1,*}$}

\newcommand{\authoraffiliations}{%
\begin{center}
\footnotesize
$^1$ Department of Mathematics and Statistics, Queen's University, Kingston, ON K7L 3N6, Canada\\
$^2$ Department of Mathematics and Computer Science, University of the Philippines Baguio, Baguio City 2600, Philippines\\
E-mail: \texttt{felicia.magpantay@queensu.ca}
\end{center}}

\makeatletter
\patchcmd{\@maketitle}
  {\ifx\@empty\authors \else \@setauthors \fi}
  {\ifx\@empty\authors \else \@setauthors \fi\authoraffiliations}
  {}{}
\makeatother

\thanks{$^*$ Corresponding author: Felicia Magpantay.}


\begin{document}

\begin{abstract}
Analysis of mathematical models in ecology and epidemiology often emphasizes asymptotic dynamics, such as stable equilibria, periodic orbits, and attractors. However, many biologically relevant trajectories exhibit long transient dynamics in which a relevant observable evolves slowly for an extended period before a comparatively rapid transition occurs. On finite time scales relevant to data collection, management, and intervention, such transient behavior can be more informative than the eventual asymptotic outcome. In this work, we study long-lasting and slowly varying transient dynamics in continuous-time systems through the framework of nonequilibrium transient centers. These centers are points in the state space near which the observable can remain nearly constant for a long time, even when the full state is not near an equilibrium. We develop finite-time escape criteria for identifying such centers and their reachable counterparts. These criteria provide practical conditions under which nearby trajectories can remain in a slow regime for a long period of time before eventually leaving it. We apply the theory to an epidemic model with delay of stability loss. In this model, trajectories may remain close to the disease-free manifold for a long time before the infectious population resurges. The framework developed here describes this behavior as slow evolution near a nonequilibrium transient center followed by finite-time escape from the slow regime.
\end{abstract}

\maketitle

\section{Introduction}

Mathematical models in ecology and epidemiology are often analyzed through their asymptotic dynamics, such as stable equilibria, periodic orbits, attractors, and their basins of attraction. This type of analysis is fundamental, but it does not always describe the behavior that is most relevant on ecological, epidemiological, or management time scales. In many systems, trajectories may spend a long time in a slowly varying regime before undergoing a rapid transition to a different dynamical regime. Such long-lasting transient behavior can be decisive for intervention and interpretation, especially when the eventual asymptotic state is reached only after a time scale much longer than the period over which data are collected, or decisions are made.

The importance of transient dynamics in ecology has been emphasized in a series of works beginning with Hastings and collaborators \cite{hastings2004transients,hastings2018transient}. These works showed that ecological systems may remain far from their asymptotic behavior for many generations, and that a sudden change in the observed dynamics need not be caused by an immediate change in environmental conditions. Long transients can arise from a variety of mechanisms, including saddle crawl-by dynamics, ghost attractors, chaotic saddles, chaotic repellers, slow-fast structure, stochastic effects, and high-dimensional or infinite-dimensional effects \cite{hastings2018transient,morozov2020long,morozov2024long}. This has important implications for ecosystem management, since a system that appears stable over a finite observation window may in fact be undergoing a long transient and eventually shift even in the absence of further change in external conditions \cite{francis2021management,sardanyes2024identifying}.

Recent work has continued to develop both theoretical and computational approaches to long transients. Morozov and collaborators studied long transients in discontinuous discrete-time population models and in models with density-dependent dispersal, showing that transient-related regime shifts and delayed extinctions can occur over very long time scales \cite{morozov2023long,morozov2024regime}. In slow-fast predator-prey systems, normal-form reductions, invariant-manifold geometry, and averaging methods have been used to analyze long oscillatory transients and to identify early warning signs of regime shifts and population collapse \cite{sadhu2022analysis,sadhu2024analysis}. Other approaches quantify transient duration through data-driven prediction methods \cite{heggerud2024model}, reaching times \cite{kittel2017timing}, and first passage times \cite{poulsen2024first}. The influence of stochasticity on ecological transients has also been studied systematically, as noise may shorten, prolong, create, or qualitatively alter transient behavior \cite{hastings2021effects}. In particular, stochasticity may shorten transients generated by ghost attractors or produce apparent long transients through noise-induced transitions between alternative states \cite{reimer2021noise}. Long transients have also been observed in spatially extended ecological models, where quasi-stable spatial patterns may persist for long periods before small perturbations produce an abrupt transition to a different pattern \cite{xu2025mechanisms}. Ghost attractors and their structures, including ghost channels and ghost cycles, have also been developed as geometric objects that organize transient dynamics \cite{koch2024ghost}. Together, these studies show that long transients are not isolated curiosities, but rather a recurring feature of nonlinear biological systems.

A complementary observable-based framework was introduced by A. Liu and Magpantay \cite{liu2022quantification} and further developed by A. Liu et al. \cite{liu2023framework}. Instead of measuring transience by the distance of a trajectory from an attractor, this framework begins with an observable $v$, such as the infectious population in an epidemic model or the prey density in a predator-prey model. The central quantity is the rate at which this observable changes along solutions. A trajectory is considered slow relative to $v$ when this rate of change is small.

This perspective is useful in applications where only one component, or one function of the full state, is biologically relevant or observable. For example, during the honeymoon period of an infectious disease, the number of infectious individuals may remain close to zero for a long time, even though other variables such as the susceptible population continue to change \cite{kharazian2020honeymoon,mclean1988measles}. Thus, the system may appear nearly stationary from the point of view of the observable, while the full state is still evolving.

Within this framework, transient points are initial conditions whose observable changes slowly for a prescribed long time but eventually leaves the slow regime. The requirement that the trajectory eventually exits the slow regime is essential because it distinguishes long transient dynamics from dynamics that remain slow forever. Transient points cluster close to points called transient centers. These are points whose neighborhoods contain initial conditions that generate arbitrarily long and arbitrarily slow transient behavior. A transient center is therefore not a transient point, but rather it organizes nearby trajectories that display long-lasting slow evolution before a finite-time transition.

The reachable version of the theory strengthens this idea by requiring that the slow regime can be approached from outside. In continuous-time systems, this is expressed by requiring finite-time escape from the slow regime both forward and backward in time \cite{liu2022quantification,liu2023framework}. This distinction is important in applications because it identifies slow regimes that trajectories can actually enter, rather than slow regimes that occur only for specially chosen initial conditions.

The theory of transient centers has been developed for continuous-time systems \cite{liu2022quantification,liu2023framework} and has also been extended to discrete-time systems \cite{pasion2026long}. In the continuous-time setting, equilibria are natural candidates because every observable is constant along an equilibrium trajectory. However, many biologically relevant long transients can occur far from equilibrium. An observed quantity may remain nearly constant even while the full state of the system moves along a nonstationary trajectory. This situation arises, for instance, during a honeymoon period, when disease prevalence appears to remain low for an extended time while other unobserved variables, such as the susceptible population, slowly evolve and eventually create conditions for a rapid transition.

The purpose of the present work is to study these nonequilibrium transient centers. We focus on nonstationary trajectories along which the observable is constant. In other words, the full state continues to move, but the observable does not change along the candidate trajectory. However, this condition alone is not enough. One must also show that nearby trajectories eventually leave a fixed slow regime in finite time. Thus, the main technical problem is to develop verifiable finite-time escape criteria near nonequilibrium candidate trajectories. We develop several such criteria based on growth estimates along individual trajectories, together with a transport argument for situations in which these estimates become valid only after a trajectory enters a suitable region of phase space.

As an application, we study an epidemic model with delay of stability loss introduced by Li et al. \cite{li2016turning}. In this model, the total host population evolves slowly, and the disease-free dynamics pass through a turning point where the disease-free state changes from transversely attracting to transversely repelling. This loss of stability does not, however, lead to an immediate departure from the disease-free manifold. Instead, trajectories may remain close to the disease-free manifold for a long time during a prolonged interepidemic period, before the infectious population grows rapidly and an outbreak occurs. Within the framework of transient centers, the prolonged interepidemic period corresponds to a long transient during which disease prevalence changes slowly, while the subsequent outbreak corresponds to a finite-time departure from a prescribed slow regime.

The paper is organized as follows. In Section~\ref{sec:preliminaries}, we recall the basic definitions of transient points, transient centers, and reachable transient centers for continuous-time systems. We also state the finite-time escape characterizations that will be used throughout the paper. In Section~\ref{sec:finite-time-escape}, we develop finite-time escape criteria for identifying nonequilibrium transient centers and explain how these criteria can be used to establish reachability. In Section~\ref{sec:application}, we apply the theory to the epidemic model of \cite{li2016turning}. We conclude in Section~\ref{sec:conclusion} with a summary of the main results.

\section{Preliminaries} \label{sec:preliminaries}

We denote the set of all positive integers by $\mathbb{N}$. For $n\in\mathbb{N}$, the $n$-dimensional real Euclidean space is denoted by $\mathbb{R}^n$, and $\|\cdot\|$ denotes the Euclidean norm on $\mathbb{R}^n$. For $x,y\in\mathbb{R}^n$, we write $x\cdot y=x^\top y$ for their Euclidean inner product. When $n=1$, we write $\mathbb{R}^1$ simply as $\mathbb{R}$ and use $|\cdot|$ to denote the absolute value of a real number. We write $\overline{\mathcal{C}}$ and $\partial\mathcal{C}$ for the closure and boundary, respectively, of a set $\mathcal{C}\subset\mathbb{R}^n$. For $j\in \{0\} \cup\mathbb{N}$, we denote by $C^j(\mathbb{R}^n,\mathbb{R})$ the space of scalar-valued functions whose partial derivatives of order at most $j$ exist and are continuous. Similarly, $C^j(\mathbb{R}^n,\mathbb{R}^n)$ denotes the space of vector-valued functions whose component functions belong to $C^j(\mathbb{R}^n,\mathbb{R})$. For a scalar-valued function $v\in C^1(\mathbb{R}^n,\mathbb{R})$, the notation $\nabla v(x)$ denotes its gradient at $x$. For a vector field $f\in C^1(\mathbb{R}^n,\mathbb{R}^n)$, the notation $Df(x)$ denotes its Jacobian matrix at $x$.

We begin by recalling the framework of transient points, transient centers, and reachable transient centers introduced in \cite{liu2022quantification} and further developed in \cite{liu2023framework}. Unless otherwise indicated, the definitions and results in this section are restated in the notation used throughout this paper and, in some cases, are presented in an equivalent form. Proofs are omitted for brevity and may be found in the cited works.

Let $f:\mathbb{R}^n\to\mathbb{R}^n$ be a $C^1(\mathbb{R}^n, \mathbb{R}^n)$ vector field. We consider the autonomous system
\begin{align}\label{eq:ode}
    x'=f(x),
\end{align}
and denote its flow by $\varphi_t$ under the following assumptions.
\begin{enumerate}[({H}1)]
    \item  For any $x(0) = \xi \in \mathbb{R}^n$, the system has a unique solution that is differentiable for all $t \in \mathbb{R}$.
\item  For any $x(0) = \xi \in \mathbb{R}^n$, the solution depends continuously on $\xi$ for all $t \in \mathbb{R}$.
\end{enumerate}
Let $v:\mathbb{R}^n\to\mathbb{R}$ be a $C^1(\mathbb{R}^n, \mathbb{R})$ function. We denote the derivative of $v$ along solutions of \eqref{eq:ode} by
\begin{align}\label{eq:Dv}
    Dv(x)=\nabla v(x)\cdot f(x).
\end{align}
Thus, if $x(t)=\varphi_t\xi$, then $\frac{d}{dt}v(x(t))=Dv(x(t))$. Therefore, $|Dv|$ measures the instantaneous speed of the observable $v$ along the flow. In this framework, a trajectory is considered slow when $|Dv|$ is small. In applications, we regard $v$ as the observable function representing a quantity of interest in the system and typically defined in terms of the system's state variables.

For $s>0$, define the quantity
\begin{align}
    T_s(\xi)=\inf\{t\geq 0:|Dv(\varphi_t \xi)|>s\}, \label{eq:Ts}
\end{align}
with the convention that $T_s(\xi)=\infty$ if the set is empty. We interpret the condition $|Dv(\varphi_t\xi)|\leq s$ as meaning that the trajectory being \textit{$s$-slow} at time $t$, while $|Dv(\varphi_t\xi)|>s$ means that it has entered the corresponding \textit{$s$-fast} regime. Accordingly, $T_s(\xi)$ records the first time at which the observable speed along the trajectory exceeds the threshold $s$.

\begin{figure}[htbp]
     \centering
     \includegraphics[width=0.7\textwidth]{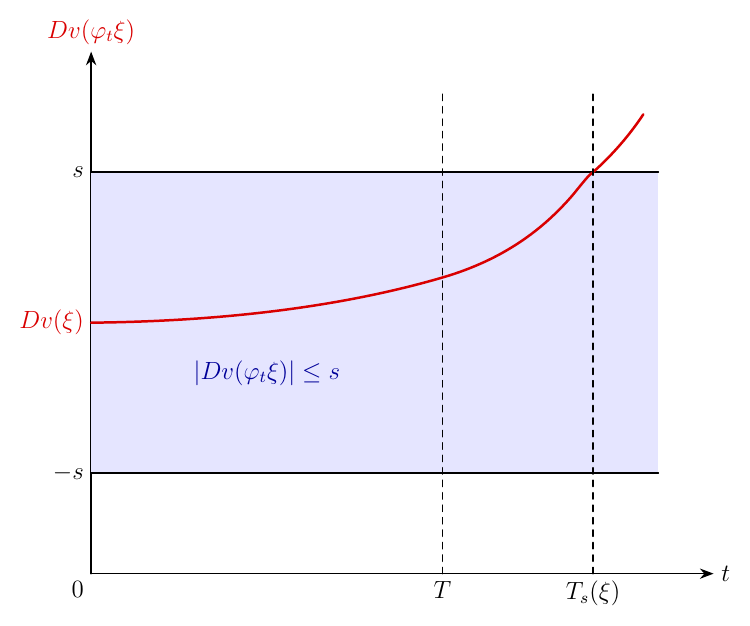}
     \caption{Illustration of a $(v,s,T)$-transient point $\xi$.}
     \label{fig:transient-point}
 \end{figure}

\begin{definition}[$(v,s,T)$-transient point]
Let $s>0$ and $T>0$. A point $\xi\in\mathbb{R}^n$ is called a \textit{$(v,s,T)$-transient point} if $T<T_s(\xi)<\infty$.
\end{definition}

The condition $T<T_s(\xi)<\infty$ implies that the observable $v$ changes slowly for at least time $T$, but the trajectory eventually leaves the $s$-slow regime, as illustrated in Figure~\ref{fig:transient-point}. We also emphasize that the condition $T_s(\xi)<\infty$ is essential since it distinguishes long transient dynamics from dynamics that remain slow forever.

\begin{definition}[$v$-transient center]
A point $x^*\in \mathbb{R}^n$ is called a $v$-transient center if there exists $S>0$ such that, for every $0<s<S$, every $T>0$, and every neighborhood $U$ of $x^*$, there exists a point $\xi\in U$ that is a $(v,s,T)$-transient point.
\end{definition}

Thus, a transient center is not a transient point. Rather, it organizes nearby initial conditions that generate arbitrarily slow dynamics for an arbitrarily long but finite time. A necessary condition for $x^* \in \mathbb{R}^n$ to be a $v$-transient center is that the observable speed vanish along the entire forward orbit of $x^*$. This led to the construction of the following collection of points
\begin{align}
    \Xi =\{x \in \mathbb{R}^n:Dv(\varphi_t x)=0 ~\text{for all}~ t \in \mathbb{R}\}. \label{eq:Xi}
\end{align}
The set $\Xi$ consists of candidate points for transient centers because the observable is constant along the corresponding trajectories. However, membership in $\Xi$ is not sufficient to be a $v$-transient center. In addition to this, we also require nearby trajectories to eventually escape the slow region.

The following criterion is often the most convenient way to determine which members of $\Xi$ are transient centers.

\begin{theorem}[Theorem 3.4 in \cite{liu2023framework}] \label{thm:finite-time criterion}
Let $x^*\in\Xi$. Then $x^*$ is a $v$-transient center if and only if there exists $S^*>0$ such that every neighborhood $U$ of $x^*$ contains a point $\xi\in U$ satisfying $T_{S^*}(\xi)<\infty$.
\end{theorem}

We shall also use the following structural properties, with the final assertion stated in terms of the closure of the full orbit through $x^*$.

\begin{theorem}[Theorems 2.3 and 3.1 in \cite{liu2023framework}] \label{thm:structure}
Let $x^*$ be a $v$-transient center. Then:
\begin{enumerate}[(a)]
    \item $Dv(x^*) =0$;

    \item $\varphi_tx^*$ is a $v$-transient center for all $t\geq 0$;
    
    \item if $Dv(\varphi_{-t}x^*)=0$ for all $t>0$, then $\varphi_t x^*$ is a
    $v$-transient center for all $t\in\mathbb{R}$;
    
    \item $x^*$ is a $(\sigma v+\beta)$-transient center for every
    $\sigma\neq 0$ and $\beta\in\mathbb{R}$;
    
    \item if, in addition, $x^*\in\Xi$, then every point in $\overline{\{\varphi_t x^*:t\in\mathbb{R}\}}$ is also a $v$-transient center.
\end{enumerate}
\end{theorem}

We now recall the reachable version of the theory. Let $\widehat{\varphi}_t$ denote the flow of the reversed system
\begin{align}
    x'=-f(x). \label{eq:reversed-system}
\end{align}
By uniqueness of solutions, whenever the original flow is defined for all $t\in\mathbb{R}$, we have $\widehat{\varphi}_t x=\varphi_{-t}x$. We define the corresponding reversed exit time by
\begin{align}
    \widehat{T}_s(\xi) = \inf\{t\geq 0 : |Dv(\widehat{\varphi}_t\xi)|>s\}, \label{eq:hatTs}
\end{align}
with the convention that $\widehat{T}_s(\xi)=\infty$ if the set is empty. Thus, $\widehat{T}_s(\xi)$ records the first time, in backward time, at which the observable speed exceeds the threshold $s$.

\begin{definition}[reachable $(v,s,T)$-transient point]
Let $s>0$ and $T>0$. A point $\xi\in\mathbb{R}^n$ is called a \textit{reachable $(v,s,T)$-transient point} if it is a $(v,s,T)$-transient point for both the original system \eqref{eq:ode} and the reversed system \eqref{eq:reversed-system}. Equivalently, $T<T_s(\xi)<\infty$ and $T<\widehat{T}_s(\xi)<\infty$.
\end{definition}

Hence, a reachable transient point remains in the $s$-slow regime for a long time both forward and backward in time, but exits the slow regime in finite time in both directions. This captures the idea that the slow regime is not generated merely by initializing near it but can also be reached dynamically from outside it.

\begin{definition}[reachable $v$-transient center]
A point $x^*\in\mathbb{R}^n$ is called a \textit{reachable $v$-transient center} if there exists $S>0$ such that, for every $0<s<S$, every $T>0$, and every neighborhood $U$ of $x^*$, there exists a point $\xi\in U$ that is a reachable $(v,s,T)$-transient point.
\end{definition}

Every reachable $v$-transient center is a $v$-transient center. The converse, however, does not necessarily hold. Reachability is a stronger condition which requires nearby trajectories to have finite-time escape both forward and backward in time.

The following result gives some of the basic properties of reachable transient centers, with the final assertion stated in terms of the closure of the full orbit through $x^*$.

\begin{theorem}[Theorems 4.4 and 4.6 in \cite{liu2023framework}]
Let $x^*$ be a reachable $v$-transient center. Then:
\begin{enumerate}[(a)]
    \item $x^*$ is a $v$-transient center for both the original system \eqref{eq:ode} and the reversed system \eqref{eq:reversed-system};

    \item $x^*\in\Xi$;

    \item $\varphi_t x^*$ is a reachable $v$-transient center for all $t\in\mathbb{R}$;

    \item $x^*$ is also a reachable $v$-transient center for the reversed system \eqref{eq:reversed-system};

    \item $x^*$ is a reachable $(\sigma v+\beta)$-transient center for every $\sigma\neq 0$ and $\beta\in\mathbb{R}$;

    \item every point in the set $\overline{\{\varphi_t x^*:t\in\mathbb{R}\}}$ is a reachable $v$-transient center.
\end{enumerate}
\end{theorem}

The reachable analogue of Theorem~\ref{thm:finite-time criterion} is the following.

\begin{theorem}[Theorem 4.8 in \cite{liu2023framework}]\label{thm:reachable finite-time criterion}
Let $x^*\in\Xi$. Then $x^*$ is a reachable $v$-transient center if and only if there exists $S^*>0$ such that every neighborhood $U$ of $x^*$ contains a point
$\xi\in U$ satisfying $T_{S^*}(\xi)<\infty$ and $\widehat{T}_{S^*}(\xi)<\infty$.
\end{theorem}

Therefore, for points in $\Xi$, reachability can be verified by finding nearby initial conditions that escape a fixed slow regime in finite time both forward and backward. In this sense, the set $\Xi$ provides the mechanism for arbitrarily long slowness, while the two finite-time escape conditions provide the forward and backward accessibility conditions that are needed for reachability.

Transient centers arising from equilibria have already been treated in \cite{liu2023framework}. Among other results, this work shows that Lyapunov stable equilibria cannot be $v$-transient centers for any $v \in C^1(\mathbb{R}^n, \mathbb{R})$ and gives criteria under which unstable equilibria are transient centers or reachable transient centers. Since the purpose of the present work is to study transient centers that arise along nonstationary trajectories, we do not develop a new theory for the equilibrium case. Instead, we use the existing equilibrium results when needed and focus our analysis on nonequilibrium candidates $x\in\Xi$ satisfying $f(x)\neq 0$.


In the next section, we develop sufficient finite-time escape criteria for points in $\Xi$ to be $v$-transient centers and explain how the same ideas can be used to establish reachable $v$-transient centers. Throughout the remainder of the paper, we write $g=Dv$ for simplicity of notation. Thus, the condition $|g|\leq s$ describes the $s$-slow regime, while $|g|>s$ describes the corresponding $s$-fast regime. If $x^*\in\Xi$, then $g(\varphi_t x^*)=0$ for all $t\in\mathbb{R}$, so the observable is constant along the trajectory of $x^*$. To prove that $x^*$ is a $v$-transient center, Theorem~\ref{thm:finite-time criterion} shows that it is enough to find, in every neighborhood of $x^*$, an initial condition  whose trajectory leaves a fixed slow regime in finite forward time. For reachable transient centers, one must verify the corresponding escape property in both time directions as stated in Theorem~\ref{thm:reachable finite-time criterion}. 

\section{Finite-time escape criteria for transient centers} \label{sec:finite-time-escape}

We shall now use the derivative of $g = Dv$ along solutions. Since $g$ measures the instantaneous speed of the observable $v$, we know that the quantity
\begin{align}
    Dg(x)=\nabla g(x)\cdot f(x) \label{eq:Dg}
\end{align}
measures how this observable speed changes along the flow. Indeed, if $x(t)=\varphi_t \xi$ and $G(t)=g(x(t))$, then $G'(t)=Dg(x(t))$.

A first escape mechanism arises when $Dg$ factors through $g$ in a neighborhood of the candidate point. More precisely, we assume that $g \in C^1(\mathbb{R}^n,\mathbb{R})$ and $Dg=\alpha g$ for some continuous function $\alpha$. Along trajectories, this gives $g'=\alpha g$. Hence, if $\alpha$ is uniformly positive in a suitable region, then small nonzero values of $|g|$ grow exponentially. This gives a simple mechanism for finite-time escape from the slow regime, as shown in the following theorem.

\begin{theorem} \label{thm:alpha criterion}
Let $x^*\in\Xi$ and $g\in C^1(\mathbb R^n,\mathbb R)$. Suppose there exist an open set $W\subset\mathbb R^n$, a continuous function
$\alpha:W\to\mathbb R$, and constants $m>0$ and $S^*>0$ such that 
\begin{enumerate}[({A}1)]
    \item for every neighborhood $V$ of $x^*$, there exists $z\in V\cap W$ such that $0<|g(z)|\leq S^*$;

    \item if $z\in W$ and $|g(z)|\leq S^*$, then $\varphi_t z\in W$ for all $0\leq t<T_{S^*}(z)$;

    \item for every $y\in W$ satisfying $|g(y)|\leq S^*$, we have 
    \begin{align*}
        Dg(y)=\alpha(y)g(y) \quad \text{and} \quad \alpha(y)\geq m. 
    \end{align*}
\end{enumerate}
Then $x^*$ is a $v$-transient center.
\end{theorem}

\begin{proof}
By Theorem~\ref{thm:finite-time criterion}, it is enough to show that every neighborhood of $x^*$ contains a point whose trajectory exits the $S^*$-slow regime in finite time. Let $V$ be an arbitrary neighborhood of $x^*$. By (A1), choose $z_0\in V\cap W$ such that $0<|g(z_0)|\leq S^*$. Let $z(t)=\varphi_t z_0$ and $G(t)=g(z(t))$. For $0\leq t<T_{S^*}(z_0)$, we have $|G(t)|\leq S^*$. Hence, $z(t)\in W$ for all such $t$ by (A2). Using (A3) and the identity $G'(t) = Dg(z(t))$, we obtain
\begin{align*}
      G'(t)=\alpha(z(t))G(t), \qquad \alpha(z(t))\geq m.
\end{align*}
Therefore,
\begin{align*}
    |G(t)| = |G(0)|\exp\left(\int_0^t \alpha(z(\tau))~d\tau\right) \geq  |G(0)|e^{mt}.
\end{align*}
Thus, $|G(t)|$ reaches $S^*$ no later than
\begin{align*}
    t_*=\frac{1}{m}\log\left(\frac{S^*}{|G(0)|}\right)<\infty.
\end{align*}
This shows that $T_{S^*}(z_0)<\infty$, proving that $x^*$ is a $v$-transient center.
\end{proof}

Condition (A3) reduces the evolution of $g$ to the scalar equation $g'=\alpha g$ along trajectories and gives a uniform lower bound on the growth rate of $|g|$. Condition (A2) ensures that the trajectory remains in the region where the estimate is valid until it leaves the slow regime, while (A1) guarantees that one can start arbitrarily close to $x^*$ with small but nonzero observable speed. Together, these assumptions give the needed escape mechanism. We now apply Theorem~\ref{thm:alpha criterion} to a simple planar system. The purpose of the example below is to illustrate how the hypotheses of the theorem are verified in practice. Once a candidate point belongs to $\Xi$, the only remaining work is to choose a region $W$ where $Dg=\alpha g$, $\alpha$ is uniformly positive, and trajectories cannot leave $W$ before leaving the slow region.

\begin{example} \label{example: planar}
  Let us consider the planar system
\begin{align}
    x'=-xy,\qquad y'=x-1, \label{eq:xy-example}
\end{align}
with observable $v(x,y)=x$. This example was one of the motivating examples in \cite{liu2022quantification} where the origin was shown directly to be a $v$-transient center. It was later revisited in \cite{liu2023framework}, where the authors showed that the candidate set for this system is
\begin{align*}
    \Xi=\{(0,y):y\in\mathbb{R}\}\cup\{(1,0)\},
\end{align*}
and that all points in $\Xi$ except the equilibrium $(1,0)$ are $v$-transient centers. Thus, the conclusion that every point $(0,y)$ is a $v$-transient center is already known. Our purpose here is not to reprove this result from scratch but to show how Theorem~\ref{thm:alpha criterion} applies to the part of the candidate set where $y<0$.

We show that every $p^*=(0,y^*)$ with $y^*<0$ is a $v$-transient center. We refer the reader to Figure~\ref{fig:example_proof}(a) for an illustration of the constructions used below. First, $\nabla g(x,y)=(-y,-x)$, and so
\begin{align*}
    Dg(x,y)=\nabla g(x,y)\cdot f(x,y)
    =(-y,-x)\cdot(-xy,x-1)
    =x(y^2-x+1).
\end{align*}
Thus, for $y\neq 0$,
\begin{align}\label{eq:alpha-example}
    Dg(x,y)=\alpha(x,y)g(x,y) 
    \quad \text{with} \quad 
    \alpha(x,y)=-\frac{y^2-x+1}{y}.
\end{align}
Fix $p^*=(0,y^*)$ with $y^*<0$. Choose $\delta>0$ such that $y^*+\delta<0$, and set
\begin{align*}
    c=-(y^*+\delta)>0,
    \quad
    W=\{(x,y): y<y^*+\delta\}=\{(x,y):y<-c\}.
\end{align*}
We choose $\delta_x\in\left(0,1/2\right]$ and let $S^*=c\delta_x$.

We now verify the hypotheses of Theorem~\ref{thm:alpha criterion}. First, every neighborhood $V$ of $p^*$ contains points $z\in V\cap W$ with $0<|g(z)|\leq S^*$. Indeed, for sufficiently small $\varepsilon> 0$, the point $z_\varepsilon =(\varepsilon,y^*)$ belongs to $V\cap W$, and $0<|g(z_\varepsilon)|=|\varepsilon y^*|<S^*$. Next, suppose $z=(x,y)\in W$ and $|g(z)|\leq S^*$. Since $y<-c$, we have $|y|\geq c$. Thus, $|xy|\leq S^* =c \delta_x$ implies that $|x|\leq \delta_x$. Using \eqref{eq:alpha-example} and the fact that $y<0$ on $W$, we get
\begin{align*}
     \alpha(x,y)= \frac{y^2-x+1}{|y|} 
     \geq \frac{y^2-|x|+1}{|y|} 
    \geq \frac{y^2+1/2}{|y|} \geq |y| \geq c.
\end{align*}
Hence, the last hypothesis holds with $m=c>0$.

It remains to verify that trajectories starting in $W$ cannot leave $W$ before leaving the $S^*$-slow regime. Let $z_0\in W$ satisfy $|g(z_0)|\leq S^*$, and write $z(t)=\varphi_t z_0=(x(t),y(t))$. Suppose, for contradiction, that the trajectory leaves $W$ before time $T_{S^*}(z_0)$. Then there exists a first exit time
\begin{align*}
     \tau=\inf\{t\in[0,T_{S^*}(z_0)):z(t)\notin W\}.
\end{align*}
Since $z_0\in W$ and $W$ is open, we have $\tau>0$, and $z(t)\in W$ for all $t \in [0, \tau)$. Hence, we have $y(t)<-c$ for $t \in [0, \tau)$. Moreover, since
$\tau<T_{S^*}(z_0)$, we obtain
\begin{align*}
    |x(t)|  = \frac{|g(z(t))|}{|y(t)|}  
    \leq \frac{c\delta_x}{|y(t)|} < \delta_x \leq \frac{1}{2}\quad \forall~ t \in [0, \tau).
\end{align*}
It follows from this that $y'(t) = x(t)-1 <-\frac{1}{2}$ and so $y(t) \leq y(0) -\frac{t}{2}$ for all $t \in [0, \tau)$. Letting $t \to \tau^{-}$ and using the continuity of $y(t)$ yields
\begin{align*}
    y(\tau)\leq y(0)-\frac{\tau}{2}<-c.
\end{align*}
Hence, $z(\tau)\in W$, contradicting the definition of $\tau$. Therefore, $\varphi_t z_0\in W$ for all $0\leq t<T_{S^*}(z_0)$.

All hypotheses of Theorem~\ref{thm:alpha criterion} are satisfied. Therefore, every point $p^*=(0,y^*)$ with $y^*<0$ is a $v$-transient center for the observable $v(x,y)=x$.

\end{example}

\begin{figure}[htbp]
\centering
\subfloat[]{%
\resizebox*{6cm}{!}{\includegraphics{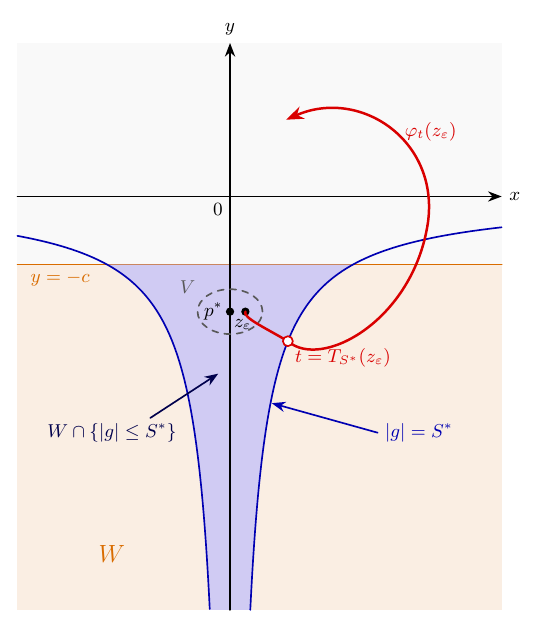}}}\hspace{5pt}
\subfloat[]{%
\resizebox*{6.3cm}{!}{\includegraphics{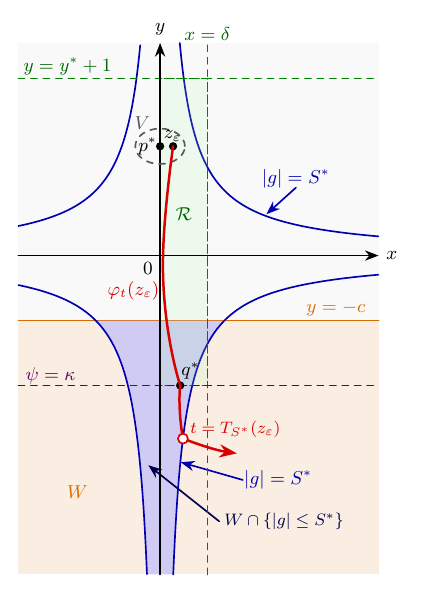}}}
\caption{Schematic of the proofs used in (a) Example~\ref{example: planar} and (b) Example~\ref{example: transport}.
}
\label{fig:example_proof}
\end{figure}


Theorem~\ref{thm:alpha criterion} gives a useful escape criterion when the exact factorization $Dg=\alpha g$ is available. We now give a slightly more flexible version in which the growth of $g$ is controlled through $g^2$. This avoids working directly with $|g|$, which is not differentiable when $g=0$. The result is pointwise in the sense that it follows one nearby trajectory and uses a comparison inequality to force eventual escape from the slow regime.

\begin{theorem} \label{thm:comparison criterion}
Let $x^*\in\Xi$ and suppose that $g^2\in C^1(\mathbb{R}^n,\mathbb{R})$. Suppose there exist a constant $S^*>0$, an open set $W\subset\mathbb{R}^n$, and a continuous function $\chi:(0,(S^*)^2]\to (0,\infty)$ such that
\begin{enumerate}[({B}1)]
    \item for every neighborhood $V$ of $x^*$, there exists $z\in V\cap W$ such that $0<|g(z)|\leq S^*$;
    
    \item if $z\in W$ and $|g(z)|\leq S^*$, then $\varphi_t z\in W$ for all $0\leq t<T_{S^*}(z)$;
    
    \item for every $y\in W$ satisfying $0<|g(y)|\leq S^*$, we have 
    \begin{align*}
        D(g^2)(y)\geq \chi(g(y)^2).
    \end{align*}
\end{enumerate}
Then $x^*$ is a $v$-transient center.
\end{theorem}

\begin{proof}
Let $V$ be an arbitrary neighborhood of $x^*$. By (B1), there exists $z_0\in V\cap W$ such that $0<|g(z_0)|\leq S^*$. Write $z(t)=\varphi_t z_0$ and $X(t)=g(z(t))^2$. We claim that $T_{S^*}(z_0)<\infty$. Suppose, to the contrary, that $T_{S^*}(z_0)=\infty$. Then $ X(t)\leq (S^*)^2$ for all $t \geq 0$. We first show that $X(t)>0$ for all $t\geq 0$. Since $X(0)>0$, define $\tau=\inf\{t>0:X(t)=0\}$. Suppose that $\tau<\infty$. Then $ 0<X(t)\leq (S^*)^2$ for all $t \in [0, \tau)$. Moreover, $z(t)\in W$ for all $t\geq 0$ since $T_{S^*}(z_0)=\infty$ by (B2). Hence, assumption (B3) implies that 
\begin{align*}
    X'(t)=D(g^2)(z(t))\geq \chi(X(t))>0 \quad \forall~t \in [0, \tau).
\end{align*}
Thus, $X$ is strictly increasing on $[0,\tau)$, and therefore $X(t)\geq X(0)>0$ for $t \in [0, \tau)$. Letting $t\to \tau^{-}$ yields $ X(\tau)\geq X(0)>0$, which contradicts the definition of $\tau$. Therefore, no such $\tau$ exists.
Hence, $0<X(t)\leq (S^*)^2$ for all $t\geq 0$. Now (B2) and (B3) imply that
\begin{align*}
     X'(t)=D(g^2)(z(t))\geq \chi(X(t)) \quad \forall~t\geq 0.
\end{align*}
Since $X(0)>0$ and $\chi$ is continuous and positive on the compact interval $[X(0),(S^*)^2]$, there exists $k>0$ such that $  \chi(r)\geq k$ for $r\in [X(0),(S^*)^2]$. Thus, $X'(t)\geq k$ and so $X(t)\geq X(0)+k t$ for all $t \geq 0$. This yields $X(t)>(S^*)^2$ for sufficiently large $t$, contradicting the assumption that $T_{S^*}(z_0)=\infty$. Thus, $T_{S^*}(z_0)<\infty$. Since $V$ was arbitrary, Theorem~\ref{thm:finite-time criterion} implies that $x^*$ is a $v$-transient center.
\end{proof}


Theorem~\ref{thm:alpha criterion} is a special case of Theorem~\ref{thm:comparison criterion}. Indeed, if $Dg=\alpha g$ and $\alpha\geq m>0$ on the relevant slow region, then
\begin{align*}
    D(g^2)=2gDg=2\alpha g^2\geq 2m g^2.
\end{align*}
Thus, Theorem~\ref{thm:comparison criterion} applies by choosing $\chi(r)=2mr$. The advantage of Theorem~\ref{thm:comparison criterion} is that it does not require an exact factorization of $Dg$ by $g$. It is enough to prove that $g^2$ increases at a controlled positive rate while the trajectory remains in the slow regime.

The previous result forces escape by controlling the growth of $g^2$. We now state a related criterion in which the role of $g^2$ is replaced by an auxiliary scalar function. This is useful when $g$ is not monotone along trajectories, but some scalar quantity must increase while the trajectory remains in the slow regime.

\begin{theorem} \label{thm:auxiliary criterion}
Let $x^*\in\Xi$. Suppose there are constants $S^*>0$, $c>0$ and $m \in \mathbb{R}$, an open set $W\subset\mathbb{R}^n$, and a function $\Lambda\in C^1(W,\mathbb{R})$ such that
\begin{enumerate}[({C}1)]
    \item for every neighborhood $V$ of $x^*$, there exists $z\in V\cap W$ such that $|g(z)|\leq S^*$;

    \item if $z\in W$ and $|g(z)|\leq S^*$, then $\varphi_t z\in W$ for all $0\leq t<T_{S^*}(z)$;

    \item for all $y\in W$ satisfying $|g(y)|\leq S^*$, we have
    \begin{align*}
        \Lambda(y) \leq m \quad \text{and} \quad D\Lambda(y)\geq c.
    \end{align*}
\end{enumerate}
Then $x^*$ is a $v$-transient center.
\end{theorem}

\begin{proof}
We again apply Theorem~\ref{thm:finite-time criterion}. Let $V$ be a neighborhood of $x^*$. By (C1), we choose $z_0\in V\cap W$ such that $|g(z_0)|\leq S^*$. We claim that $T_{S^*}(z_0)<\infty$. Suppose, to the contrary, that $T_{S^*}(z_0)=\infty$. This implies that $|g(\varphi_t z_0)|\leq S^*$ for all $t \geq 0$, and we also have $\varphi_t z_0\in W$ for all $t \geq 0$ by assumption (C2). In other words, $\varphi_t z_0\in W\cap\{|g|\leq S^*\}$ for all $t \geq 0$. Define $Y(t) = \Lambda(\varphi_t z_0)$. Using (C3), we have $Y(t) \leq m$ and 
\begin{align*}
    Y'(t) = D\Lambda(\varphi_t z_0) \geq c \quad \forall~t\geq 0.
\end{align*}
Thus, $Y(t) \geq Y(0) + ct$ which implies that $Y(t) \to \infty$ as $t \to \infty$, a contradiction. Therefore, we must have $T_{S^*}(z_0)<\infty$, completing the proof.
\end{proof}

Note that assumption (C1) differs from assumptions (A1) of Theorem~\ref{thm:alpha criterion} and (B1) of Theorem~\ref{thm:comparison criterion} in that it permits the selected point $z$ to satisfy $g(z)=0$. This is possible because the proof of Theorem~\ref{thm:auxiliary criterion} does not require the observable speed $g$ to be initially nonzero. Moreover, under the hypotheses of Theorem~\ref{thm:auxiliary criterion}, the point $x^*$ cannot belong to $W$. Indeed, since $x^*\in\Xi$, we have $g(\varphi_t x^*)=0$ for all $t\in\mathbb{R}$. Hence, $T_{S^*}(x^*)=\infty$. Suppose that $x^*\in W$. 
Applying assumptions (C2) and (C3) along the forward trajectory of $x^*$ leads to the same contradiction between the upper bound on $\Lambda$ and the strict growth of $\Lambda$ obtained in the proof of Theorem~\ref{thm:auxiliary criterion}. Thus, $x^* \notin W$. On the other hand, assumption (C1) implies that $x ^* \in \overline{W \cap \{|g| \leq S^*\}}$. In particular, $x^*\in\overline{W}$. Since $W$ is open and $x^*\notin W$, it follows that $x^*\in\partial W$. Thus, the set $W$ approaches the candidate point $x^*$ without containing it. This observation will become apparent in Section~\ref{sec:application}, where Theorem~\ref{thm:auxiliary criterion} is applied directly to an epidemic model.


The escape criteria given in the theorems above require that the nearby initial condition already lies in a region where a suitable growth estimate is valid. In many applications, however, an initial point chosen arbitrarily close to a candidate $x^*\in\Xi$ may not immediately satisfy this condition. It might first have to traverse an initial segment of phase space before entering the region where the growth estimate holds. The following elementary lemma isolates this preliminary transport step. It does not itself give a criterion for finding transient centers but instead provides a way to show that a trajectory that remains $S^*$-slow must reach a prescribed level set of an auxiliary function $\psi$ in finite time. Once that level set is attained, one can apply one of the immediate escape criteria to conclude that the trajectory eventually leaves the slow regime. Thus, the lemma serves as a bridge that extends the applicability of Theorems~\ref{thm:alpha criterion}, \ref{thm:comparison criterion}, and \ref{thm:auxiliary criterion} to situations where the required growth estimates are only guaranteed away from the candidate point.

\begin{lemma}\label{lem:transport}
Let $S^*>0$, $\mathcal R\subset\mathbb{R}^n$ be open, 
$\psi\in C^1(\mathbb{R}^n,\mathbb{R})$, and $\kappa\in\mathbb{R}$. For
$z\in\mathbb{R}^n$, define
\begin{align*}
     \tau_\kappa(z)=\inf\{t\geq 0:\psi(\varphi_t z)\geq \kappa\}
\end{align*}
with the convention that $\inf\varnothing=\infty$. If
\begin{enumerate}[({L}1)]
\item there is $z_0\in\mathcal R$ satisfying  $|g(z_0)|\leq S^*$ and $\psi(z_0)<\kappa$,

\item $\varphi_t z_0\in\mathcal R$ for all  $0\leq t<\min\{T_{S^*}(z_0),\tau_\kappa(z_0)\}$, and

\item there exists $c>0$ such that $ D\psi(y)=\nabla\psi(y)\cdot f(y)\geq c$ for all $y\in\mathcal R$ satisfying $|g(y)|\leq S^*$ and $\psi(y)<\kappa$,
\end{enumerate}
then
\begin{align*}
     \min\{T_{S^*}(z_0),\tau_\kappa(z_0)\} \leq \frac{\kappa-\psi(z_0)}{c}.
\end{align*}
\end{lemma}

\begin{proof}
Set $t_*=\min\{T_{S^*}(z_0),\tau_\kappa(z_0)\}$ and suppose for contradiction that
$t_*>\frac{\kappa-\psi(z_0)}{c}$. Choose $t\geq 0$ such that $\frac{\kappa-\psi(z_0)}{c}<t<t_*$. For every $\tau\in[0,t]$, we have $\tau<t_*$. This implies that $\tau<T_{S^*}(z_0)$ and $\tau<\tau_\kappa(z_0)$, and consequently, $|g(\varphi_\tau z_0)|\le S^*$ and $\psi(\varphi_\tau z_0)<\kappa$. By (L2), we know that $\varphi_\tau z_0\in\mathcal R$. Thus, (L3) yields
\begin{align*}
    \frac{d}{d\tau}\psi(\varphi_\tau z_0)=D\psi(\varphi_\tau z_0)\ge c \quad \forall~ \tau\in[0,t].
\end{align*}
Integrating from $0$ to $t$, we obtain $\psi(\varphi_t z_0)\ge \psi(z_0)+ct>\kappa$. This contradicts the fact that $t<\tau_\kappa(z_0)$. Hence, the assumption was false, and we conclude that $t_*\leq \frac{\kappa-\psi(z_0)}{c}$, completing the proof.
\end{proof}

The lemma is designed for situations where a nearby trajectory starts in a neighborhood of a candidate point $x^*\in\Xi$ but not yet inside a region $W$ where the growth estimates of Theorems~\ref{thm:alpha criterion}, \ref{thm:comparison criterion}, or \ref{thm:auxiliary criterion} are valid. For such a trajectory, we can choose an auxiliary function $\psi$ whose time derivative is bounded below by a positive constant as long as the trajectory remains slow. The lemma then guarantees that within a finite time the trajectory must either leave the $S^*$-slow regime, thereby giving the required finite-time escape, or reach a level set where $\psi$ attains a value that ensures the trajectory has entered the region $W$. Once inside $W$, the hypotheses of Theorem~\ref{thm:alpha criterion}, Theorem~\ref{thm:comparison criterion}, or Theorem~\ref{thm:auxiliary criterion} become applicable, and the trajectory is forced to exit the $S^*$-slow regime in finite time. Thus, by combining the lemma with any of these criteria, we can prove that $x^*$ is a $v$-transient center even when the growth estimates are not initially available arbitrarily close to $x^*$. The next example shows how we apply this strategy.

\begin{example} \label{example: transport}
We now return to the planar system in Example~\ref{example: planar} and consider the remaining candidate points $p^*=(0,y^*)$ with $y^*>0$. Near such a point, the function $\alpha$ in \eqref{eq:alpha-example} is negative since
\begin{align*}
    \alpha(x,y)=-\frac{y^2-x+1}{y}<0
\end{align*}
for $y>0$ and $x$ sufficiently small. Thus, the observable speed $g=Dv$ is not initially growing in the sense required by Theorem~\ref{thm:alpha criterion}. Nevertheless, trajectories near the line $x=0$ move downward because $y'=x-1<0$ whenever $x$ is sufficiently small. The idea is therefore to use Lemma~\ref{lem:transport} to transport a nearby slow trajectory from the region $\{y>0\}$ into a region $\{y<0\}$ where Theorem~\ref{thm:alpha criterion} applies. The arguments below are best understood by consulting Figure~\ref{fig:example_proof}(b).

Fix $p^*=(0,y^*)$ with $y^*>0$. Let $c>0$, $\delta\in(0,1/2]$ and set $S^*=c\delta$. We shall transport nearby points from the region $\{(x,y): y>0\}$ into the region
\begin{align*}
    W=\{(x,y):y<-c\},
\end{align*}
where the estimates used in Example~\ref{example: planar} apply. Thus, any point $q^*\in W$ with $0<|g(q^*)|\leq S^*$ leaves the $S^*$-slow regime in finite time. 


We now show that such a point $q^*$ is reached from arbitrarily close to
$p^*=(0,y^*)$, unless the trajectory has already escaped. Let
\begin{align*}
    \psi(x,y)=-y \quad \text{and} \quad \kappa=2c,
\end{align*}
so that the level set $\psi=\kappa$ is the line $y=-2c$. Define the region
\begin{align*}
    \mathcal R=\{(x,y):0<x<\delta,~ -2c<y<y^*+1\}.
\end{align*}
Note that on $\mathcal R$, we have
\begin{align*}
    D\psi(x,y)=1-x\geq 1-\delta\geq \frac{1}{2}.
\end{align*}
Let $V$ be an arbitrary neighborhood of $p^*$. We choose $z_\varepsilon=(\varepsilon,y^*)\in V\cap \mathcal R$ with $\varepsilon>0$ sufficiently small so that
\begin{align}
    0<|g(z_\varepsilon)|=\varepsilon y^*<S^* 
    \quad \text{and} \quad 
    \varepsilon e^{4c(y^*+2c)}<\frac{\delta}{2}. \label{eq:varepsilon_choice}
\end{align}
We claim that the trajectory starting from $z_\varepsilon$ remains in $\mathcal R$ until it either leaves the $S^*$-slow regime or reaches the level $y=-2c$.

We first show that the trajectory cannot leave $\mathcal R$ before either escaping the $S^*$-slow regime or reaching the level $y=-2c$, except possibly through the boundary $x=\delta$. Since $x(0)=\varepsilon>0$ and $x'=-xy$, we have
\begin{align*}
     x(t)=\varepsilon\exp\left(-\int_0^t y(s)~ds\right)>0
\end{align*}
for as long as the solution exists. Hence, the trajectory cannot cross the boundary $x=0$. Moreover, before the level $y=-2c$ is reached, the trajectory cannot exit through the lower boundary $y=-2c$. Also, while the trajectory remains in $\mathcal R$, we have $x(t)<\delta\leq \frac{1}{2}$, and therefore $y'(t)=x(t)-1\leq -\frac{1}{2}$. This implies that $y(t)$ is decreasing as long as the trajectory remains in $\mathcal R$. Since $y(0)=y^*<y^*+1$, the trajectory cannot exit through the upper boundary $y=y^*+1$. Consequently, if the trajectory leaves $\mathcal R$ before either escaping the $S^*$-slow regime or reaching the level $y=-2c$, then it must leave through the boundary $x=\delta$.

Let $t_1$ be the first time at which $x(t_1)=\delta$, assuming that this time occurs before the trajectory leaves the $S^*$-slow regime and before it reaches the level $y=-2c$, that is, $t_1<\min\{T_{S^*}(z_\varepsilon),\tau_\kappa(z_\varepsilon)\}$. This means that the trajectory has not reached the level $y=-2c$ at any time $t \in [0, t_1]$, and so $y(t) > -2 c$ for $t \in [0, t_1]$. On $[0,t_1]$, we therefore have $x'(t)=-x(t)y(t)\leq 2cx(t)$, which implies $x(t)\leq \varepsilon e^{2ct}$ for all $t \in [0, t_1]$. Also, since $t_1$ is the first time at which $x(t)=\delta$, we have $x(t)<\delta\leq \frac{1}{2}$ and so $y'(t)=x(t)-1\leq -\frac{1}{2}$ for $t \in [0, t_1)$. Integrating yields $y(t_1)\leq y^*-\frac{t_1}{2}$. Since $y(t_1)>-2c$, we obtain $t_1<2(y^*+2c)$. Consequently,
\begin{align*}
     x(t_1)\leq \varepsilon e^{2ct_1} < \varepsilon e^{4c(y^*+2c)} < \frac{\delta}{2}
\end{align*}
using \eqref{eq:varepsilon_choice}. This contradicts $x(t_1)=\delta$. Therefore, the trajectory cannot leave $\mathcal R$ before either escaping the $S^*$-slow regime or reaching $y=-2c$.

The hypotheses of Lemma~\ref{lem:transport} are therefore satisfied for the initial point $z_\varepsilon$, with derivative lower bound $1-\delta$. Hence,
\begin{align*}
    \min\{T_{S^*}(z_\varepsilon),\tau_\kappa(z_\varepsilon)\} 
    \leq \frac{\kappa-\psi(z_\varepsilon)}{1-\delta} 
    = \frac{2c+y^*}{1-\delta} <\infty.
\end{align*}
We claim that $T_{S^*}(z_\varepsilon)<\infty$. If $T_{S^*}(z_\varepsilon)=\infty$, then the preceding inequality implies that $\tau_\kappa(z_\varepsilon)<\infty$. Let $q^*=\varphi_{\tau_\kappa(z_\varepsilon)}z_\varepsilon$. By continuity of $\psi(\varphi_tz_\varepsilon)$ and the definition of $\tau_\kappa$, we have $\psi(q^*)=\kappa$, and hence the $y$-coordinate of $q^*$ is $-2c$. Thus, $q^*\in W$. Moreover, since $T_{S^*}(z_\varepsilon)=\infty$, the trajectory has not left the $S^*$-slow regime, and so $|g(q^*)|\leq S^*$. Also, $x(t)>0$ along the trajectory, and therefore $|g(q^*)|=|x(\tau_\kappa)y(\tau_\kappa)|>0$. Hence, $0<|g(q^*)|\leq S^*$. By the same estimates used in Example~\ref{example: planar} on the region $W=\{(x,y):y<-c\}$, we have $T_{S^*}(q^*) < \infty$. Since $q^*=\varphi_{\tau_\kappa(z_\varepsilon)}z_\varepsilon$, it follows that the original trajectory starting from $z_\varepsilon$ also leaves the $S^*$-slow regime in finite time. This contradicts $T_{S^*}(z_\varepsilon)=\infty$. Since $V$ was arbitrary, every neighborhood of $p^*$ contains a point $z_\varepsilon$ such that $T_{S^*}(z_\varepsilon)<\infty$ and, Theorem~\ref{thm:finite-time criterion} implies that $p^*$ is a $v$-transient center.

\end{example}


Observe that a common hypothesis in all of the theorems we have proven so far is the following invariance assumption:
\begin{align*}
    z\in W, \quad |g(z)| \leq S^* \implies \varphi_t z \in W \quad \forall~t \in [0, T_{S^*} (z)).
\end{align*}
This means that as long as the trajectory has not yet entered the fast regime, it must remain inside $W$. Without this, our proofs only show that the trajectory cannot remain forever in $W \cap \{|g| \leq S^*\}$. There are two possibilities: either $|g| > S^*$ which is the desired escape, or $\varphi_tz \notin W$ while still having $|g(\varphi_t z)| \leq S^*$. The second possibility does not prove the finite-time escape from the $S^*$-slow regime. It only proves exit from the auxiliary region $W$. In applications, verifying this invariance assumption may be nontrivial. If $W$ is chosen merely as a bounded neighborhood of the candidate point, a trajectory may exit through its boundary while still satisfying $|g| \leq S^*$. Then our results do not apply unless we can show that such slow exit from $W$ is impossible. Nevertheless, our theorems are useful when $W$ is chosen so that the slow part of the dynamics is positively invariant, or at least invariant up to the first exit from the $S^*$-slow regime. For instance, one may, as in Examples~\ref{example: planar} and~\ref{example: transport}, choose $W$ to be defined by inequalities whose boundary vector field points inward while $|g| \leq S^*$.

We conclude this section by explaining how the preceding escape criteria can be used to establish reachability. By Theorem~\ref{thm:reachable finite-time criterion}, it is not enough to prove forward escape for one nearby point and backward escape for another. One must find, in every neighborhood of $x^*$, a single point $z$ such that $T_{S^*}(z)<\infty$ and $\widehat T_{S^*}(z)<\infty$. Thus, to obtain reachable $v$-transient centers from the preceding results, we apply the same type of escape argument in both time directions. The forward escape is proved using the original flow $\varphi_t$, while the backward escape is proved by applying the corresponding criterion to the reversed flow $\widehat\varphi_t z=\varphi_{-t}z$. In practice, this means that the hypotheses of Theorems~\ref{thm:alpha criterion}, \ref{thm:comparison criterion}, \ref{thm:auxiliary criterion}, or Lemma~\ref{lem:transport} must be verified twice: once for the original system and once for the reversed system, with both arguments applying to the same chosen initial point $z$. This distinction is important because a point may be a $v$-transient center for the forward system and also for the reversed system without being a reachable $v$-transient center (see, for instance, Example 4.5 in \cite{liu2023framework}). Reachability requires the forward and backward finite-time escape mechanisms to be realized by the same nearby trajectory.

\section{Application to an epidemic model with delay of stability loss} \label{sec:application}

We now apply the criteria developed in the previous section to an epidemic model studied by Li et al. \cite{li2016turning}. Their work considers an SIR epidemic model in which the total host population varies slowly. The model keeps track of the susceptible population $S$, the infectious population $I$, and the total population $N$. The recovered population $R$ is determined by $R= N-S-I$.

\subsection{The epidemic model}


In the absence of disease, the total population is assumed to satisfy $N'=\varepsilon q(N)$ where $0<\varepsilon\ll 1$. Thus, $\varepsilon$ represents a small intrinsic growth rate of the host population. The function $q$ describes density-dependent population growth with the assumption that $q$ is concave and satisfies $q''(N)<0$ and $q(0)=q(N_*)=0$ for some $N_*>0$. Consequently, $q(N)>0$ for $0<N<N_*$, so $N_*$ plays the role of the carrying capacity. The disease transmission is described by the incidence term $h(S,N)I$ where the function $h$ is smooth, increasing in $S$, and satisfies $h(0,N)=0$. The parameters and functions appearing in the model are summarized in Table~\ref{tab:Li-parameters}. Defining $a=\zeta+\mu+\gamma$, the model of Li et al.~\cite{li2016turning} is
 \begin{align}\label{eq:Li-model}
 \left\{\begin{aligned}
     S' &= \zeta N+\varepsilon q(N)-h(S,N)I-(\zeta+p)S,\\
     I' &= h(S,N)I-aI,\\
     N' &= \varepsilon q(N)-\mu I.
 \end{aligned} \right.
 \end{align}
The biologically feasible region is
\begin{align*}
    \mathcal D=\{(S,I,N)\in\mathbb R^3:~S\geq0,~I\geq0,~N\geq0,~S+I\leq N\leq N_*\}.
\end{align*}
This region is positively invariant for the system \eqref{eq:Li-model}. The disease-free plane $\{(S, I, N) : I=0\}$ is also invariant, reflecting the biological fact that the disease cannot arise if no infectious individuals are initially present.

\begin{table}[t]
    \centering
    \begin{tabular}{|p{0.15\textwidth}|p{0.75\textwidth}|}
        \hline
        \multicolumn{1}{|c|}{Symbol}
        & \multicolumn{1}{c|}{Description} \\
        \hline
        $\varepsilon$
        & small positive parameter determining the slow time scale of host population growth \\
        \hline
        $q(N)$
        & density-dependent host population growth function, assumed to satisfy $q''(N)<0$, $q(0)=q(N_*)=0$, and $q(N)>0$ for $0<N<N_*$ \\
        \hline
        $N_*$
        & carrying capacity of the host population \\
        \hline
        $h(S,N)$
        & disease transmission function, assumed to be smooth, increasing in $S$, and to satisfy $h(0,N)=0$ \\
        \hline
        $\zeta$
        & natural death rate of the host population \\
        \hline
        $p$
        & vaccination rate of susceptible individuals \\
        \hline
        $\gamma$
        & recovery rate of infectious individuals \\
        \hline
        $\mu$
        & disease-induced death rate \\
        \hline
        $a$
        & total removal rate from the infectious class, given by
        $a=\zeta+\mu+\gamma$ \\
        \hline
    \end{tabular}
    \caption{Parameters and functions appearing in  model~\eqref{eq:Li-model}.}
    \label{tab:Li-parameters}
\end{table}

Restricting \eqref{eq:Li-model} to the disease-free plane gives
\begin{align}\label{eq:disease-free-reduced}
\left\{\begin{aligned}
    S' &= \zeta N+\varepsilon q(N)-(\zeta + p)S,\\
    N' &= \varepsilon q(N).
\end{aligned}\right.
\end{align}
Thus, for $0<\varepsilon\ll1$, the total population $N$ evolves on a slow time scale. In the singular limit $\varepsilon=0$, the disease-free dynamics have the critical manifold
\begin{align*}
    Z_0 = \left\{(S,0,N): S=\frac{\zeta N}{\zeta + p}\right\}.
\end{align*}
The transverse dynamics in the infectious direction are governed by the factor $h(S,N)-a$ since  $I'=(h(S,N)-a)I$. Therefore, along $Z_0$, the stability of the disease-free state in the infectious direction is determined by
\begin{align*}
    \lambda(N) =h\left(\frac{\zeta N}{\zeta + p},N\right)-a.
\end{align*}
Following \cite{li2016turning}, it is assumed that $h\left(\frac{\zeta N}{\zeta + p},N\right)$ is increasing in $N$, and that there exists a unique $N_c\in(0,N_*)$ such that $ h\left(\frac{\zeta N_c}{\zeta + p},N_c\right)=a$. Writing $S_c=\frac{\zeta}{\zeta + p}N_c$, the point $(S_c,0,N_c)$ is the turning point. For $N<N_c$, we have $\lambda(N)<0$, so small infectious perturbations decay. For $N>N_c$, we have $\lambda(N)>0$, so small infectious perturbations grow. Hence, the disease-free critical manifold changes from transversely attracting to transversely repelling as the disease-free dynamics pass through the turning point.

This change of transverse stability is the source of delay of stability loss. A trajectory with a small infectious population may be attracted toward the disease-free manifold while $\lambda(N)<0$. As the slow disease-free dynamics carry the trajectory past the turning point, $\lambda(N)$ becomes positive. However, the trajectory does not necessarily leave the disease-free manifold immediately. Instead, it may remain close to the disease-free manifold for a long time before the infectious population grows rapidly. This phenomenon has been studied in the context of dynamical bifurcations and singular perturbation problems. In particular, Neishtadt \cite{neishtadt2009stability} studied delay of stability loss for dynamical bifurcations, while W. Liu \cite{liu2000exchange,liu2006geometric} developed exchange-lemma methods for singular perturbation problems with turning points.

In the epidemiological interpretation, the long portion of the trajectory near the disease-free manifold corresponds to an interepidemic period, while the rapid departure from the disease-free manifold corresponds to an epidemic outbreak. Li et al. \cite{li2016turning} used this mechanism to prove the existence of large-amplitude relaxation oscillations under suitable parameter conditions. 

From the perspective of transient centers, delay of stability loss provides a natural setting in which long-lasting and slowly varying transient dynamics can occur. In the present model, trajectories may remain near the disease-free manifold for a long time before eventually departing from it. This long residence time corresponds to slow transient behavior with respect to the observable of interest, while the eventual departure corresponds to finite-time escape from a prescribed slow regime. Thus, although delay of stability loss is a more specific phenomenon in slow-fast systems, the theory of transient centers and reachable transient centers provides a useful viewpoint for describing the transient behavior it produces near the disease-free manifold. We now make this viewpoint precise using the observable $v(S,I,N)=I$.


\subsection{$v$-transient centers}

We first prove a local result near the disease-free equilibrium at carrying capacity $E_+=(S_*,0,N_*)$, where $S_*=\frac{\zeta}{\zeta + p} N_*$. Note that this point is stable for the dynamics restricted to the disease-free plane $\{I=0\}$. However, by the turning point assumption, it is unstable in the transverse infectious direction. Indeed, since $h\left(\frac{\zeta N}{\zeta + p},N\right)$ is increasing in $N$ and crosses the value $a$ at $N=N_c\in(0,N_*)$, we have
\begin{align*}
    h(S_*,N_*)-a = h\left(\frac{\zeta N_*}{\zeta + p},N_*\right)-a >0.
\end{align*}
Thus, nearby trajectories with small $I>0$ experience growth in the infectious direction once they remain sufficiently close to $E_+$. Throughout, we assume that $\varepsilon>0$.

\begin{lemma} \label{lem:estimates}
Let $H(S,N)=h(S,N)-a$,
\begin{align*}
    F(S,N) =
    \begin{bmatrix}
        \zeta N+\varepsilon q(N)-(\zeta + p)S\\
        \varepsilon q(N)
    \end{bmatrix}, 
    \quad \text{and} \quad  
    Q(S,N) =
    \begin{bmatrix}
        -h(S,N)\\
        -\mu
    \end{bmatrix}.
\end{align*}
There exist positive constants $\eta$, $b$, $\beta$, a neighborhood $\Omega\subset\mathbb R^2$ of $(S_*,N_*)$, and a function $\mathcal{L}\in C^1(\mathbb R^2,\mathbb R)$ such that
\begin{enumerate}[(a)]
    \item  $H(S,N)\geq 2\eta$ for all $(S,N)\in\overline{\Omega}$,
    \item $\nabla \mathcal{L}(S,N)\cdot F(S,N)\leq -b$ for all $(S,N)\in\partial\Omega$, and 
    \item $\left| \nabla \mathcal{L}(S,N)\cdot Q(S,N) \right| \leq \beta$ for all $(S,N)\in\overline{\Omega}$.
\end{enumerate}
\end{lemma}

\begin{proof}
Since $H(S_*,N_*)>0$ and $H$ is continuous, there exists a neighborhood $U_H\subset \mathbb R^2$ of $(S_*,N_*)$ and a constant $\eta>0$ such that $H(S,N)\geq 2\eta$ for all $(S,N)\in U_H$. Next, consider the disease-free reduced system \eqref{eq:disease-free-reduced} which can be written as $\begin{bmatrix}
        S'\\
        N'
    \end{bmatrix} = F(S,N)$.
Note that the point $(S_*,N_*)$ is an equilibrium of this system. Its Jacobian matrix at $(S_*,N_*)$ is
\begin{align*}
     A= \begin{bmatrix}
        -(\zeta + p) & \zeta+\varepsilon q'(N_*)\\
        0 & \varepsilon q'(N_*)
    \end{bmatrix}
\end{align*}
whose eigenvalues are $-(\zeta + p) <0$ and $\varepsilon q'(N_*)<0$. Note that $q'(N_*)<0$ since $q''(N)<0$, $q(0)=q(N_*)=0$, and $q(N)>0$ for $0<N<N_*$. Since $A$ is Hurwitz, there exists a positive definite symmetric matrix $P$ satisfying $A^\top P+PA=-I_2$ where $I_2$ denotes the $2\times 2$ identity matrix (see, for instance, Theorem 4.6 in \cite{khalil2002nonlinear}). Define
\begin{align*}
      y= \begin{bmatrix}
        S-S_*\\
        N-N_*
    \end{bmatrix} \quad \text{and} \quad \mathcal{L}(S,N)=y^\top Py.
\end{align*}
Near $(S_*,N_*)$, the reduced disease-free system can be written as $y'=Ay+R(y)$ where $R(y)=o(\|y\|)$ as $y\to0$. Along solutions,
\begin{align*}
    \nabla \mathcal{L}(S,N)\cdot F(S,N)
    &= 2y^\top P y'\\
    &= 2y^\top P(Ay+R(y))\\
    &= y^\top(A^\top P+PA)y+2y^\top PR(y)\\
    &= -\|y\|^2+2y^\top PR(y).
\end{align*}
Since $R(y)=o(\|y\|)$, there exists a neighborhood $U_R\subset\mathbb R^2$ of $(S_*,N_*)$ such that $|2y^\top PR(y)| \leq \frac{1}{2}\|y\|^2$ for all $(S,N)\in U_R$. This yields 
\begin{align}
    \nabla \mathcal{L}(S,N)\cdot F(S,N) \leq -\frac{1}{2}\|y\|^2  \label{eq:nabla_l}
\end{align}
on $U_R$. Now set $U=U_H\cap U_R$. Choose $\varrho>0$ sufficiently small so that
\begin{align}
    \Omega=\{(S,N)\in\mathbb R^2:\mathcal{L}(S,N)<\varrho\} \label{eq:Omega}
\end{align}
satisfies $\overline{\Omega}\subset U$. This is possible because $P$ is positive definite, so the sublevel sets of $\mathcal{L}(S,N)=y^\top Py$ shrink to $(S_*,N_*)$ as $\varrho\to0$. Since $\overline{\Omega}\subset U_H$, we obtain $H(S,N)\geq 2\eta$ for any $(S,N)\in\overline{\Omega}$. This proves (a). Moreover,
\begin{align*}
    \partial\Omega = \{(S,N)\in\mathbb R^2:\mathcal{L}(S,N)=\varrho\},
\end{align*}
and so $y\neq0$ on $\partial\Omega$. Since $\partial\Omega$ is compact and $\|y\|$ is continuous, there exists $r>0$ such that $\|y\|\geq r$ for all $(S,N)\in\partial\Omega$. Setting $b = \frac{1}{2}r^2 >0$, we have from \eqref{eq:nabla_l} that
\begin{align*}
     \nabla \mathcal{L}(S,N)\cdot F(S,N)\leq -\frac{1}{2}\|y\|^2 \leq -\frac{1}{2} r^2 = -b
\end{align*}
for all $(S,N)\in\partial\Omega$. This proves (b). Finally, define 
\begin{align*}
    B(S,N) = \left| \nabla \mathcal{L}(S,N)\cdot Q(S,N) \right|.
\end{align*}
Note that $B$ is bounded on $\overline{\Omega}$ since $B$ is continuous. Therefore, there exists $\beta>0$ such that $B(S, N) \leq \beta$ for all $(S,N)\in\overline{\Omega}$. This proves (c), and the lemma follows.
\end{proof}

\begin{figure}[htbp]
    \centering
    \includegraphics[width=0.9\linewidth]{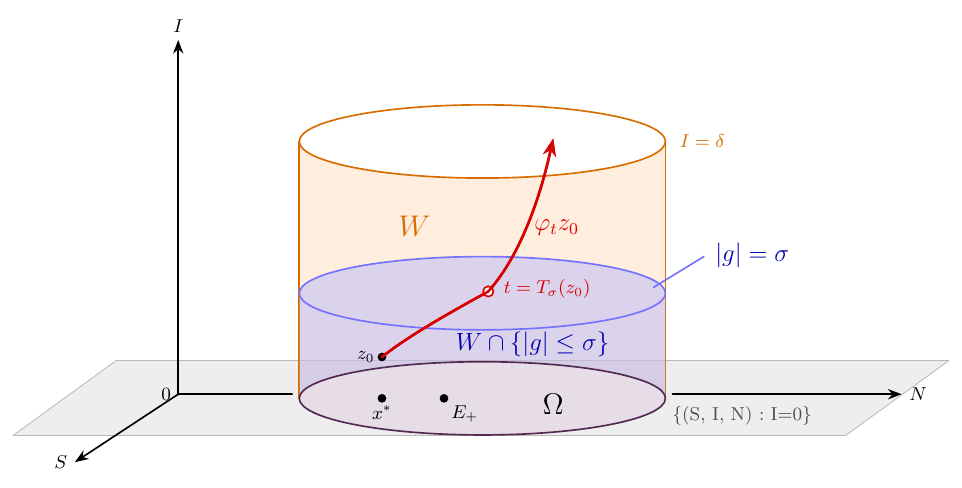}
    \caption{Illustration of the local construction used in the proof of Theorem~\ref{thm:Li-transient}. The orange cylinder represents the set $W=\Omega\times(0,\delta)$ above the disease-free plane $\{I=0\}$, while the blue region represents the $\sigma$-slow region $W\cap\{|g|\leq\sigma\}$. An initial point $z_0$ with $0<|g(z_0)|<\sigma$ remains in $W$ until the trajectory $\varphi_t z_0$ reaches the boundary $|g|=\sigma$ at the exit time $T_\sigma(z_0)$.}
    \label{fig:cylinder}
\end{figure}
    
\begin{theorem} \label{thm:Li-transient}
    There exists an open set $W\subset\mathbb R^3$ with $W\subset\{I>0\}$ and $E_+ \in \overline{W}$ such that every point in $\overline{W}\cap\{I=0\}$ is a $v$-transient center. 
\end{theorem}

\begin{proof}
    Before we proceed with the proof, we mention that the geometric content of the arguments below is illustrated in Figure~\ref{fig:cylinder}. For the observable $v(S, I, N) = I$, we have $g = Dv =I' = H(S, N) I$. It is easy to check that points on the disease-free plane $\{I=0\}$ belong to our set of candidate $v$-transient centers $\Xi$ defined in \eqref{eq:Xi}. Throughout the proof, we use the constants $\eta,b,\beta$, the open set $\Omega\subset\mathbb R^2$, and the function $\mathcal{L}$ constructed in Lemma~\ref{lem:estimates}. Choose $\delta >0$ and then choose $\sigma>0$ sufficiently small so that
    \begin{align}
        \frac{\sigma}{2\eta}<\delta < \frac{b}{\beta}. \label{eq:sigma_delta}
    \end{align}
     Define
    \begin{align}
        W = \{(S,I,N):(S,N)\in\Omega,~0<I<\delta \}. \label{eq:W}
    \end{align}
    Note that $W$ is open in $\mathbb R^3$, satisfies $W\subset\{I>0\}$, and $E_+\in\overline{W}$. We prove that every point in $\overline W\cap\{I=0\}$ is a $v$-transient center by applying Theorem~\ref{thm:auxiliary criterion}. In this application, the role of the abstract threshold $S^*>0$ in Theorem~\ref{thm:auxiliary criterion} is played by the constant $\sigma>0$.

    The function $\Lambda(S,I,N)=\log I$ is $C^1$ on $W$. Along solutions of the full system \eqref{eq:Li-model}, we have
    \begin{align*}
        D\Lambda(S,I,N) = \frac{I'}{I} = H(S,N).
    \end{align*}
    By Lemma~\ref{lem:estimates}(a), we have $H(S,N)\geq 2\eta$ for all $(S,N)\in\overline{\Omega}$. Hence, 
    \begin{align*}
        D\Lambda(S,I,N)\geq 2\eta \quad \forall~(S,I,N)\in W.
    \end{align*}

    Now let $z_0\in W$ satisfy $|g(z_0)|\leq \sigma$ and write
    \begin{align*}
        z(t)=\varphi_tz_0=(S(t),I(t),N(t)).
    \end{align*}
    We claim that $z(t)\in W$ for all $t \in [0, T_\sigma(z_0))$. Suppose, to the contrary, that there exists $t_1<T_\sigma(z_0)$ such that $z(t_1)\notin W$. Since $z_0\in W$ and $W$ is open, there is a first exit time
    \begin{align*}
        \tau=\inf\{t\in[0,t_1]:z(t)\notin W\}.
    \end{align*}
    Then $0<\tau\leq t_1<T_\sigma(z_0)$, $z(t) \in W$ for $0\leq t<\tau$, and $z(\tau) \in \partial W$. From the definition of $W$ in \eqref{eq:W}, the trajectory can exit $W$ only through one of the following boundary components: $I=0$, $I=\delta$, or $(S,N)\in\partial\Omega$. First, it is clear that it cannot exit through $I=0$. Next, we show that it also cannot exit through $I=\delta$. Since $\tau<T_\sigma(z_0)$, we have $|g(z(\tau))|\leq \sigma$. Suppose that $I(\tau)=\delta$. Because $(S(t), N(t)) \in \Omega$ for $t < \tau$, the limit $(S(\tau),N(\tau))$ as $t\to \tau^{-}$ must lie in the closure $\overline{\Omega}$ by continuity of solutions. Using Lemma~\ref{lem:estimates}(a), we have
    \begin{align*}
        |g(z(\tau))| =  H(S(\tau),N(\tau))I(\tau) \geq 2\eta\delta.
    \end{align*}
    But $\sigma$ was chosen so that $\frac{\sigma}{2\eta}<\delta$ as required in \eqref{eq:sigma_delta}. Hence, $|g(z(\tau))|>\sigma$ contradicting $\tau<T_\sigma(z_0)$. Thus, $I(\tau)\neq\delta$. It remains to rule out exit through $\partial\Omega$. Suppose $(S(\tau),N(\tau))\in\partial\Omega$. Using Lemma~\ref{lem:estimates}(a) and the fact that $\tau<T_\sigma(z_0)$, we get
    \begin{align*}
        2\eta I(\tau) \leq  H(S(\tau),N(\tau))I(\tau) =  |g(z(\tau))| \leq \sigma.
    \end{align*}
    This implies $I(\tau)\leq \frac{\sigma}{2\eta}<\delta$. Now, for the full system \eqref{eq:Li-model}, the $(S,N)$-equations can be written as
    \begin{align*}
        \begin{bmatrix}
        S'\\
        N'
    \end{bmatrix} = F(S,N)+I Q(S,N).
    \end{align*}
    Therefore,
    \begin{align*}
        \frac{d}{dt}\mathcal{L}(S(t),N(t))\bigg|_{t=\tau}
    &= \nabla\mathcal{L}(S(\tau),N(\tau))\cdot F(S(\tau),N(\tau))\\
    &\qquad \quad
    + I(\tau)\nabla\mathcal{L}(S(\tau),N(\tau))\cdot Q(S(\tau),N(\tau)).
    \end{align*}
    By Lemma~\ref{lem:estimates}(b)--(c) and the choice of $\delta$ from \eqref{eq:sigma_delta}, we obtain
    \begin{align*}
        \frac{d}{dt}\mathcal{L}(S(t),N(t))\bigg|_{t=\tau}
        \leq -b+\beta I(\tau)  < -b+\beta \delta <0.
    \end{align*}
    Note that $\mathcal{L}(S(t),N(t))<\varrho$ for $0\leq t<\tau$ while $\mathcal{L}(S(\tau),N(\tau))=\varrho$. 
    Therefore, $\mathcal{L}(S(t),N(t))$ reaches the boundary level $\varrho$ from below at time $\tau$. This is impossible if $\frac{d}{dt}\mathcal{L}(S(t),N(t))\bigg|_{t=\tau}<0$. Hence, the trajectory cannot exit through $\partial\Omega$. We have ruled out all possible exit boundaries. Therefore, no such first exit time $\tau<T_\sigma(z_0)$ exists and consequently, $z(t)\in W$ for all $t \in [0, T_\sigma(z_0))$. This verifies hypothesis (C2) of Theorem~\ref{thm:auxiliary criterion}.

    Now suppose that $z\in W$ and $|g(z)|\leq\sigma$. Recall that we have $I\leq \frac{\sigma}{2\eta}$ and so
    \begin{align*}
         \Lambda(z)=\log I \leq \log\left(\frac{\sigma}{2\eta}\right).
    \end{align*}
    Thus, condition (C3) in Theorem~\ref{thm:auxiliary criterion} holds with $c=2\eta$ and $m=\log\left(\frac{\sigma}{2\eta}\right)$. 

    It remains to verify condition (C1). Let $x^*\in\overline W\cap\{I=0\}$ and let $V$ be any neighborhood of $x^*$. Since $x^*\in\overline W$, there exists a sequence $\{z_n\}_{n \in \mathbb{N}} \subset W$, where $z_n=(S_n,I_n,N_n)$ and $z_n\to x^*$. Since $x^*\in\{I=0\}$, we have $I_n\to0$. Moreover, because $z_n\in W$, we have $(S_n,N_n)\in\Omega$ and $I_n>0$. By Lemma~\ref{lem:estimates}(a), $H(S_n,N_n)\geq 2\eta>0$. Hence, $g(z_n)=H(S_n,N_n)I_n>0$. Since $\overline{\Omega}$ is compact and $H$ is continuous on $\overline{\Omega}$, there exists $k>0$ such that $|H(S,N)|\leq k$ for all $(S,N)\in\overline{\Omega}$. Thus, $0<g(z_n)\leq k I_n\to0$. Therefore, for $n$ sufficiently large, $z_n\in V\cap W$ and $|g(z_n)|\leq \sigma$. This verifies hypothesis (C1).
    

    All hypotheses of Theorem~\ref{thm:auxiliary criterion} are therefore satisfied
    for every $x^*\in\overline W\cap\{I=0\}$. Hence, every point in $\overline W\cap\{I=0\}$ is a $v$-transient center.
\end{proof}

The preceding theorem gives a local result near the disease-free equilibrium $E_+$. In particular, it shows that all disease-free points whose $(S,N)$-coordinates lie in the closed base $\overline{\Omega}$ of $W$ are $v$-transient centers. We now extend this local conclusion to the biologically feasible disease-free region
\begin{align*}
    \mathcal D_0 = \{(S,0,N):0\leq S\leq N\leq N_*\} \subset \mathcal D
\end{align*}
by using the structural properties of transient centers. The key observation is that once a trajectory in $\mathcal D_0$ enters the local set $\overline W\cap\{I=0\}$, the property of being a $v$-transient center propagates along the trajectory.

\begin{theorem} \label{thm:Li-transientcenters}
    Every point in the set $\mathcal D_0$ is a $v$-transient center for system \eqref{eq:Li-model}.
\end{theorem}

\begin{proof}
    We first consider a point $x_0=(S_0,0,N_0)\in\mathcal D_0$ with $N_0>0$. Since the disease-free plane $\{I=0\}$ is invariant, the trajectory of $x_0$ remains in this plane. Writing $\varphi_t x_0=(S(t),0,N(t))$, the pair $(S(t),N(t))$ satisfies the disease-free reduced system \eqref{eq:disease-free-reduced} initialized at $(S_0, N_0)$.
    It was shown in \cite{li2016turning} that $(S(t),N(t))\to (S_*,N_*)$ as $t \to \infty$. Since $\Omega$ is a neighborhood of $(S_*,N_*)$, there exists $T\geq0$ such that $(S(T),N(T))\in\Omega$. Therefore, 
    \begin{align*}
        \varphi_T x_0=(S(T),0,N(T))\in \overline W\cap\{I=0\}.
    \end{align*}
    By the preceding theorem, $\varphi_T x_0$ is a $v$-transient center. Moreover, since $\{I=0\}$ is invariant and $Dv=g=H(S,N)I$, we have $Dv(\varphi_{-t}(\varphi_Tx_0))=0$ for all $t>0$. Therefore, by Theorem~\ref{thm:structure}(c), $\varphi_t(\varphi_Tx_0)$ is a $v$-transient center for all $t \in \mathbb{R}$. In particular, $x_0=\varphi_{-T}(\varphi_Tx_0)$ is a $v$-transient center.

    It remains to consider the origin. Let $W^u(0,0)$ denote the unstable manifold of the origin for the disease-free reduced system \eqref{eq:disease-free-reduced}. We regard this manifold as a subset of the full phase space through the embedding $(S,N)\mapsto (S,0,N)$. Choose any point
    \begin{align*}
        y_0=(S_0,0,N_0)\in W^u(0,0)\cap\mathcal D_0
    \end{align*}
    with $N_0>0$. The existence of such a point follows from the Stable Manifold Theorem (see Section 2.7 in \cite{perko2001differential} for a complete statement of the theorem). Indeed, the Jacobian of the disease-free reduced system at $(0,0)$ is
    \begin{align*}
     A_0= \begin{bmatrix}
        -(\zeta + p) & \zeta+\varepsilon q'(0)\\
        0 & \varepsilon q'(0)
    \end{bmatrix}.
    \end{align*}
    Since $q'(0)>0$, the origin is a hyperbolic saddle. Its unstable eigenvalue is $\lambda_u=\varepsilon q'(0)>0$ and the corresponding unstable eigenvector satisfies $S=\frac{\zeta +\lambda_u}{\zeta + p+\lambda_u}N$. Because $0<\frac{\zeta+\lambda_u}{\zeta + p+\lambda_u}<1$, this unstable direction points into the region $\{(S,N):0\leq S\leq N\leq N_*\}$. Since the unstable manifold is tangent to this eigendirection at the origin, a sufficiently small segment of its positive branch lies in the interior of $\{(S,N):0\leq S\leq N\leq N_*\}$. Hence, the positive branch of $W^u(0,0)$ contains points with $N>0$ inside $\mathcal D_0$. By the argument above, $y_0$ is a $v$-transient center and so $y_0 \in \Xi$. Moreover, since $y_0\in W^u(0,0)$, its backward orbit converges to the origin, that is, $\varphi_{-t}y_0\to (0,0,0)$ as $t \to \infty$. Therefore, $(0,0,0)\in \overline{\{\varphi_t y_0:t\in\mathbb R\}}$. Applying Theorem~\ref{thm:structure}(e), we conclude that the origin is also a $v$-transient center.

    Therefore, every point of $\mathcal D_0$ is a $v$-transient center.
\end{proof}

\subsection{Reachable $v$-transient centers}

We now turn to reachable $v$-transient centers. To make the construction precise, let $\varphi_t^0$ denote the flow of the reduced disease-free system \eqref{eq:disease-free-reduced} on the $(S,N)$-plane. Recall that $(0,0)$ is a hyperbolic saddle for this reduced system. Let $W^u_+(0,0)$ denote the branch of its unstable manifold tangent at $(0,0)$ to the eigenvector with positive $S$ and $N$ components. Choose any point $y_\Gamma\in W^u_+(0,0)\cap\{(S,N):0<S<N<N_*\}$ and define $(S_\Gamma(t),N_\Gamma(t)) =\varphi_t^0(y_\Gamma)$ for $t \in \mathbb{R}$. The choice of $y_\Gamma$ only changes the parametrization by a time translation. We define the set
\begin{align}
     \Gamma_0 = \{(S_\Gamma(t),0,N_\Gamma(t)):t\in\mathbb R\}. \label{eq:Gamma_0}
\end{align}
Thus, $\Gamma_0$ is the disease-free heteroclinic orbit from $(0,0,0)$ to $E_+$. In particular, $(S_\Gamma(t),0,N_\Gamma(t))\to (0,0,0)$ as $t\to -\infty$ and $(S_\Gamma(t),0,N_\Gamma(t))\to E_+$ as $t\to \infty$. This orbit is illustrated in the $(S,N)$-plane in Figure~\ref{fig:transient}. Since $I=0$ on $\Gamma_0$, it is easy to see that $\Gamma_0\subset \Xi$. We show in what follows that the entire trajectory $\Gamma_0$ is composed of reachable $v$-transient centers.

\begin{figure}[htbp]
    \centering
    \includegraphics[width=0.7\linewidth]{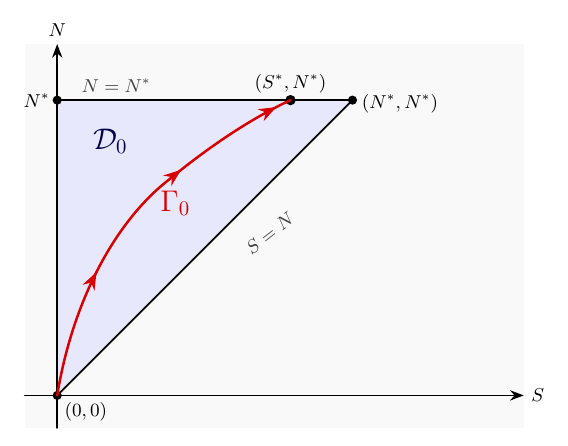}
    \caption{Schematic illustration of the biologically feasible disease-free set $\mathcal D_0$ and the heteroclinic orbit $\Gamma_0$ in the $(S,N)$-plane.}
    \label{fig:transient}
\end{figure}

Let us now consider the reversed system associated with model \eqref{eq:Li-model}. We denote its flow by $\widehat\varphi_t=\varphi_{-t}$. First, we show that points of $\Gamma_0$ are $v$-transient centers for this reversed flow. One could prove this statement by following the same strategy used in Theorems~\ref{thm:Li-transient} and~\ref{thm:Li-transientcenters}. Here, however, we present an alternative approach based on the theory of equilibria as $v$-transient centers. In particular, we use the following result of A. Liu et al. \cite{liu2023framework}, which transfers the property of being a $v$-transient center from a hyperbolic equilibrium to trajectories in $\Xi$ that converge to it. This gives the reader a complementary approach for identifying transient centers through the local structure of an equilibrium.

\begin{theorem}[Theorem 5.8 in \cite{liu2023framework}]\label{thm:Liu58}
Let $v\in C^2(\mathbb R^n,\mathbb R)$. Suppose that $0$ is a hyperbolic equilibrium of \eqref{eq:ode}, where $f\in C^1(\mathbb R^n,\mathbb R^n)$, and let $A=Df(0)$. Let $E^u$ denote the generalized unstable eigenspace of $A$. If $x^*\in\Xi$, $A^\top \nabla v(0)\not\perp E^u$, and $\varphi_tx^*\to 0$ as $t\to\infty$, then both $0$ and $x^*$ are $v$-transient centers.
\end{theorem}

We apply Theorem~\ref{thm:Liu58} to the reversed system. In this setting, convergence to the origin in forward time for the reversed flow $\widehat \varphi_t$ is equivalent to convergence to the origin in backward time for the original flow $\varphi_t$. This is precisely the behavior of points on $\Gamma_0$ since $\Gamma_0$ is the disease-free heteroclinic orbit that approaches the origin as $t\to-\infty$ and approaches $E_+$ as $t\to \infty$.

\begin{theorem}
    Every point of $\Gamma_0$ is a $v$-transient center for the reversed system associated with model \eqref{eq:Li-model}.
\end{theorem}

\begin{proof}
    Let $E_0 =(0,0,0)$. We apply Theorem~\ref{thm:Liu58} to the reversed system at the equilibrium $E_0$. The linearization of system \eqref{eq:Li-model} around $E_0$ is given by 
    \begin{align*}
        A= \begin{bmatrix}
            -(\zeta + p) & 0 & \zeta+\varepsilon q'(0)\\
                0 & -a & 0\\
                0 & -\mu & \varepsilon q'(0)
    \end{bmatrix}. 
    \end{align*}
    Hence, the linearization of the reversed system at $E_0$ is $\widehat A = -A$, whose eigenvalues are $\zeta + p>0$, $a>0$, and $-\varepsilon q'(0)<0$. If $\widehat E^u$ is the generalized unstable eigenspace of $\widehat A$, then $\widehat E^u \not\subset \{I=0\}$, and so $\widehat E^u$ contains a vector $w$ with a nonzero $I$-component, $w_I$. As a consequence, we have
    \begin{align*}
         \widehat A^\top \nabla v(E_0)\cdot w = (0,a,0) \cdot w = a w_I \neq 0.
    \end{align*}
    Therefore, $\widehat A^\top \nabla v(E_0) \not\perp \widehat E^u$. Now let $x^*\in \Gamma_0$. By definition of $\Gamma_0$, $\widehat\varphi_t x^* = \varphi_{-t}x^* \to E_0$ as $t \to \infty$. It is now immediate from Theorem~\ref{thm:Liu58} that $x^*$ is a $v$-transient center for the reversed system. 
\end{proof}

We are now ready to prove the reachability of points in $\Gamma_0$.

\begin{theorem} \label{thm:Li-reachable-transient-centers}
    Every point of $\Gamma_0$ is a reachable $v$-transient center for system \eqref{eq:Li-model}.
\end{theorem}


\begin{proof}
First, recall that in the proof of Theorem~\ref{thm:Li-transient}, we constructed an open set $W=\{(S,I,N):(S,N)\in\Omega,~0<I<\delta\}$ near the equilibrium $E_+$, together with a threshold $\sigma>0$ such that every point $z_0\in W$ satisfying $|g(z_0)|\leq \sigma$ leaves the $\sigma$-slow region in finite forward time. Note that the same argument applies, with a minor modification, to the two-sided set
\begin{align*}
     W^\pm=\{(S,I,N):(S,N)\in\Omega,~0<|I|<\delta\}.
\end{align*}
Indeed, $\Lambda(S,I,N)=\log |I|$ is $C^1$ on $W^\pm$, and along solutions
\begin{align*}
    D\Lambda(S,I,N) = \frac{I'}{I} = H(S,N).
\end{align*}
Thus, $D\Lambda\geq 2\eta$ on $W^\pm$. Moreover, we have $H\geq 2\eta>0$ on $\Omega$, and hence $|g(S,I,N)|=H(S,N)|I|$. Therefore, whenever $z\in W^\pm$ satisfies $|g(z)|\leq\sigma$, we have $|I|\leq \frac{\sigma}{2\eta}<\delta$. The only change in the invariance argument is that, on the boundary $(S,N)\in\partial\Omega$, one estimates
\begin{align*}
      I~\nabla\mathcal{L}(S,N)\cdot Q(S,N)
    \leq |I|~|\nabla\mathcal{L}(S,N)\cdot Q(S,N)|
    \leq \beta |I|.
\end{align*}
Consequently,
\begin{align*}
    \frac{d}{dt}\mathcal{L}(S(t),N(t))\bigg|_{t=\tau}
    \leq -b+\beta |I(\tau)| < -b+\beta\delta <0.
\end{align*}
Thus, trajectories cannot leave $W^\pm$ through $\partial\Omega$ before leaving the $\sigma$-slow region. The remaining boundary components, $I=0$ and $|I|=\delta$, are ruled out as before. Hence, every point $z_0\in W^\pm$ satisfying $|g(z_0)|\leq\sigma$ leaves the $\sigma$-slow region in finite forward time, that is, $T_\sigma(z_0)<\infty$.


Let $x^* \in \Gamma_0$. Since $\varphi_t x^* \to E_+$ as $t \to \infty$, we may choose $T>0$ such that $\varphi_T x^*=(S(T),0,N(T))$ satisfies $(S(T),N(T))\in\Omega$. Note that the map $z\mapsto (S(\varphi_T z),N(\varphi_T z))$ is continuous by continuous dependence of solutions on initial data. Since $\Omega$ is open, there exists a neighborhood $U$ of $x^*$ such that $(S(\varphi_T z),N(\varphi_T z))\in\Omega$ for all $z \in U$. Let $z=(S_0,I_0,N_0)\in U$ with $I_0\neq0$. If $z$ exits the $\sigma$-slow region before time $T$, then $T_\sigma(z)<\infty$. Otherwise, we have $|g(\varphi_t z)|\leq \sigma$ for all $t \in [0, T]$. In particular, $|g(\varphi_T z)|\leq \sigma$. Since $(S(\varphi_T z),N(\varphi_T z))\in\Omega$ and $H\geq 2\eta$ on $\Omega$, we get
\begin{align*}
    0<|I(\varphi_Tz)|=\frac{|g(\varphi_Tz)|}{H(S(\varphi_Tz),N(\varphi_Tz))} \leq \frac{\sigma}{2\eta} <\delta.
\end{align*}
Thus, $\varphi_T z\in W^\pm$ and we know that $T_\sigma(\varphi_T z) < \infty$. Therefore, we still have $T_\sigma(z) < \infty$.

Additionally, we know from the previous theorem that $x^*$ is a $v$-transient center for the reversed system associated with \eqref{eq:Li-model}. Hence, by Theorem~\ref{thm:finite-time criterion} applied to the reversed system, there exists $\widehat\sigma>0$ such that every neighborhood of $x^*$ contains a point $z$ satisfying $\widehat T_{\widehat\sigma}(z) < \infty$. 

Now let $\sigma_{\min}=\min\{\sigma,\widehat\sigma\}$ and suppose that $V$ is an arbitrary neighborhood of $x^*$. Since $g(x^*)=0$ and $g$ is continuous, there is a neighborhood $V_0$ of $x^*$ such that $V_0\subset V\cap U$ and $|g(z)|<\sigma_{\min}$ for all $z \in V_0$. Applying our argument from the previous paragraph to the neighborhood $V_0$, we can pick $z = (S_0, I_0, N_0) \in V_0$ such that $\widehat T_{\widehat\sigma}(z)<\infty$. It is easy to check that $I_0 \neq 0$. Indeed, if $I_0=0$, then the invariance of $\{I=0\}$ yields $I(\widehat\varphi_t z)=0$ for all $t \geq 0$ and hence $g(\widehat\varphi_t z)=0$ for all $t \geq 0$. This would imply $\widehat T_{\widehat\sigma}(z)=\infty$, a contradiction. Thus, $I_0\neq0$. Since $z \in U$ and $I_0 \neq 0$, we have $T_\sigma (z) < \infty$ as well. Therefore, $T_{\sigma_{\min}}(z) \leq T_\sigma (z) < \infty$ and $\widehat T_{\sigma_{\min}}(z) \leq\widehat T_{\widehat\sigma}(z)<\infty$. Thus, every neighborhood $V$ of $x^*$ contains a point $z$ whose forward and reversed exit times from the $\sigma_{\min}$-slow regime are both finite. We now conclude from Theorem~\ref{thm:reachable finite-time criterion} that $x^*$ is a reachable $v$-transient center. 
\end{proof}

The preceding result shows that every point of $\Gamma_0$ is a reachable $v$-transient center. It is natural to ask whether the converse is true, namely, whether $\Gamma_0$ exhausts all reachable $v$-transient centers contained in the set $\mathcal D_0$. We conjecture that for system \eqref{eq:Li-model}, the reachable $v$-transient centers in $\mathcal D_0$ are precisely the points of $\Gamma_0$.

Proving this converse appears to be more delicate than proving reachability of points on $\Gamma_0$. The argument above uses the special two-sided geometry of $\Gamma_0$. In forward time, its points approach $E_+$, where $H(S,N)>0$ gives a forward escape mechanism, while in backward time they approach the origin, where the reversed flow argument gives a backward escape mechanism. For a point of $\mathcal D_0\setminus \Gamma_0$, one would need to rule out the existence of nearby perturbations that realize both finite forward and finite backward escape. This is a nonlocal question since it depends not only on the local behavior near the point, but also on the global behavior of its forward and backward disease-free orbits and on how nearby trajectories interact with the boundary of the feasible region and with the hypersurface $H(S,N)=0$. We therefore leave the characterization of all reachable $v$-transient centers as a direction for future work.

\subsection{Numerical example}


Before concluding this section, we present a numerical illustration of the reachability of a nonequilibrium point on $\Gamma_0$ with respect to the observable $v=I$. We use the functional forms considered by Li et al. \cite{li2016turning},
\begin{align*}
    q(N)=N\left(1-\frac{N}{N_*}\right),
    \quad
    h(S,N)=\frac{\beta S}{K+S},
\end{align*}
together with the parameter values from Case 1 of their numerical simulations, summarized in Table~\ref{tab:parameter-values}. 
\begin{table}[t]
    \centering
    \begin{tabular}{|c|c|c|c|}
        \hline
        \multicolumn{1}{|c|}{Parameter}
        & \multicolumn{1}{c|}{Value}
        & \multicolumn{1}{c|}{Parameter}
        & \multicolumn{1}{c|}{Value} \\
        \hline
        $\zeta$ & $0.2$
        & $p$ & $0.01$ \\
        \hline
        $\mu$ & $0.048$
        & $\beta$ & $1$ \\
        \hline
        $\gamma$ & $0.75$
        & $K$ & $0.1$ \\
        \hline
        $N_*$ & $400$
        & $\varepsilon$ & $10^{-4}$ \\
        \hline
    \end{tabular}
    \caption{Parameter values used in the numerical illustration, taken from \cite{li2016turning}.}
    \label{tab:parameter-values}
\end{table}
For these choices, Li et al. \cite{li2016turning} numerically observed that the solution starting from 
\begin{align}
    z_0 =(S_0, I_0, N_0)=(35,0.09555,67) \label{eq:initialcondition}
\end{align}
approaches a stable relaxation oscillation.

We first obtain a numerical approximation of $\Gamma_0$ by integrating the reduced system \eqref{eq:disease-free-reduced} from a point close to the origin on the positive branch of its unstable manifold. We then integrate the full system \eqref{eq:Li-model} forward from the initial condition $z_0$ defined in \eqref{eq:initialcondition} and write
\begin{align*}
     z(t)=\varphi_tz_0=(S(t),I(t),N(t)).
\end{align*}
Over the chosen search interval, we identify the time $t_0$ at which the Euclidean distance from $z(t)$ to the numerical approximation of $\Gamma_0$ is minimized. We denote by $x^\dagger=(S^\dagger,0,N^\dagger)\in\Gamma_0$ a corresponding closest point on the numerical approximation of $\Gamma_0$. For the simulation considered here, we have
\begin{align*}
    t_0 &\approx7.1149\times10^5, \\
     z(t_0) &\approx \left(10.48647613,\, 1.51103919\times10^{-8},\, 11.01054519 \right), \\
      x^\dagger &\approx \left( 10.48526627,\, 0,\, 11.00927479 \right).
\end{align*}
Consequently,
\begin{align*}
     \left\|z(t_0)-x^\dagger\right\|\approx1.75\times10^{-3},
\end{align*}
indicating that the trajectory starting from $z_0$ passes close to the nonequilibrium point $x^\dagger\in\Gamma_0$.

Figure~\ref{fig:trajectory} displays the infectious population on a time interval surrounding the passage near $x^\dagger$. The blue curve represents the evolution leading to the passage near $x^\dagger$, while the red curve represents the subsequent evolution. The black marker and vertical dotted line indicate the selected time $t_0$. The figure shows that a trajectory initialized away from $\Gamma_0$ can approach a nonequilibrium point of $\Gamma_0$, enter a regime in which the infectious population is very small, and subsequently leave this regime as the next epidemic outbreak develops. 

\begin{figure}[htbp]
    \centering
    \includegraphics[width=0.8\linewidth, height=6cm]{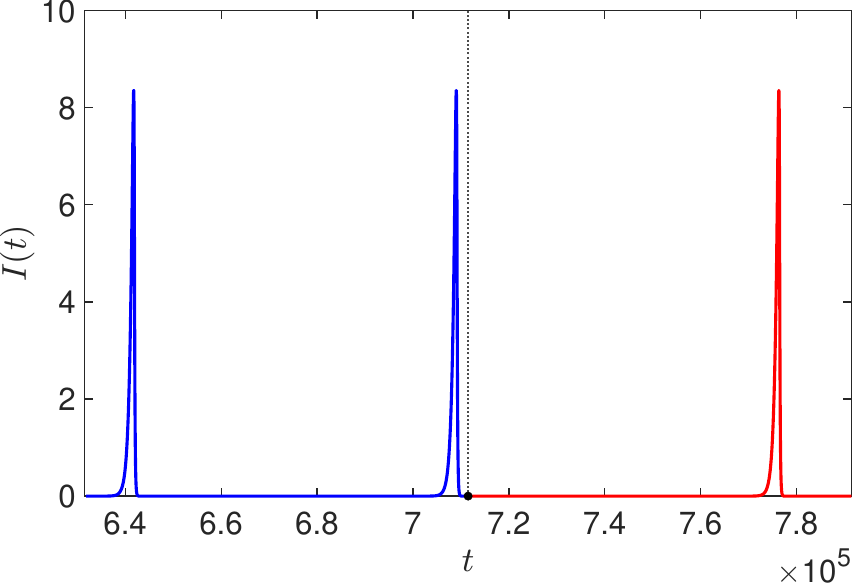}
    \caption{Evolution of the infectious population along a trajectory that reaches a neighborhood of a nonequilibrium reachable $v$-transient center $x^\dagger\in\Gamma_0$.}
    \label{fig:trajectory}
\end{figure}

\section{Conclusion}\label{sec:conclusion}

In this paper, we developed a framework for identifying nonequilibrium transient centers, namely, candidate points that lie on nonstationary trajectories along which the relevant observable remains constant. Such points describe situations in which the observable is constant along a nonstationary trajectory and may remain nearly constant for a long time on nearby trajectories, even though the full state of the system continues to evolve. This distinction is important in applications where slow behavior depends on the observable and need not arise from proximity to an equilibrium.

We established several finite-time escape criteria based on the evolution of $g=Dv$, $g^2$, or an auxiliary scalar function along nearby trajectories. These results identify mechanisms that force a trajectory to remain slow for a finite period before eventually entering a prescribed fast regime. We also introduced a transport argument for situations in which the relevant growth estimates become available only after the trajectory has first reached another region of phase space. The same ideas were extended to reachable transient centers by requiring forward and backward escape to occur along the same nearby trajectory.

Lastly, we applied the theory to an epidemic model with delay of stability loss. Taking the infectious population as the observable, we showed that every point in the biologically feasible disease-free set is a $v$-transient center and identified a heteroclinic orbit consisting of reachable $v$-transient centers. In this setting, the prolonged evolution near the disease-free manifold represents a slowly varying interepidemic phase, while departure from that manifold corresponds to the onset of rapid epidemic growth. These results show that nonequilibrium transient centers provide a useful framework for describing long transients associated with nonstationary trajectories and observable-dependent slow behavior.


\section*{Disclosure statement}
The authors report that there are no competing interests to declare.

\section*{Funding}
This work was supported by the Natural Sciences and Engineering Research Council of Canada (NSERC) Discovery Grant Program.

\end{document}